\documentclass[11pt,a4paper]{article}
\usepackage{multirow}
\usepackage{epsfig}
\usepackage{epstopdf}
\usepackage{stmaryrd}
\usepackage[T1]{fontenc}
\usepackage{geometry}
\usepackage{amsbsy,amsmath,latexsym,amsfonts, epsfig, color, authblk, amssymb, graphics, bm}
\usepackage{epsf,slidesec,epic,eepic}
\usepackage{fancybox}
\usepackage{fancyhdr}
\usepackage{setspace}
\usepackage{nccmath}
\usepackage{cases}
\usepackage{extarrows}
\usepackage{multirow}
\usepackage{algorithm}
\usepackage{algorithmic}
\usepackage{rotating}
\usepackage{emptypage}
\usepackage{graphicx}
\DeclareGraphicsExtensions{.eps,.ps,.jpg,.bmp}
\usepackage{mathrsfs}
\usepackage[colorlinks, citecolor=blue]{hyperref}

\newtheorem{theorem}{Theorem}[section]
\newtheorem{lemma}{Lemma}[section]

\newtheorem{definition}{Definition}[section]
\newtheorem{example}{Example}[section]

\newtheorem{assumption}{Assumption}[section]
\newtheorem{remark}{Remark}[section]

\newenvironment{proof}{{\noindent \bf Proof:}}{\hfill$\Box$\medskip}

\definecolor{lred}{rgb}{1,0.8,0.8}
\definecolor{lblue}{rgb}{0.8,0.8,1}
\definecolor{dred}{rgb}{0.6,0,0}
\definecolor{dblue}{rgb}{0,0,0.5}
\definecolor{dgreen}{rgb}{0,0.5,0.5}

 \title{A proximal dual semismooth Newton method for computing zero-norm penalized QR estimator}

 \author{Dongdong Zhang\footnote{(mathzdd@mail.scut.edu.cn) School of Mathematics, SCUT, Guangzhou, China.}
 \ \ Shaohua Pan\footnote{(shhpan@scut.edu.cn) School of Mathematics, South China University of Technology,
 China.}\ \ {\rm and}\ \
 Shujun Bi\footnote{(bishj@scut.edu.cn) School of Mathematics, South China University of Technology, China.}}

\begin{document}

 \maketitle

 \begin{abstract}
  This paper is concerned with the computation of the high-dimensional
  zero-norm penalized quantile regression estimator, defined as a global
  minimizer of the zero-norm penalized check loss function.
  To seek a desirable approximation to the estimator, we reformulate this
  NP-hard problem as an equivalent augmented Lipschitz optimization problem,
  and exploit its coupled structure to propose a multi-stage convex relaxation
  approach (MSCRA\_PPA), each step of which solves inexactly a weighted
  $\ell_1$-regularized check loss minimization problem with a proximal
  dual semismooth Newton method. Under a restricted strong convexity condition,
  we provide the theoretical guarantee for the MSCRA\_PPA
  by establishing the error bound of each iterate to the true estimator and
  the rate of linear convergence in a statistical sense. Numerical comparisons
  on some synthetic and real data show that
  MSCRA\_PPA not only has comparable even better estimation performance,
  but also requires much less CPU time.
 \end{abstract}


\section{Introduction}\label{sec1}

  Sparse penalized regression has become a popular approach for high-dimensional
  data analysis. In the past two decades, many classes of sparse penalized
  regressions have been developed by imposing a suitable penalty term on
  the least squares loss such as the bridge penalty in \cite{Frank93},
  Lasso in \cite{Tibshirani96}, SCAD in \cite{Fan01}, elastic net in \cite{Zou05},
  adaptive lasso by \cite{Zou06}, and so on. We refer to the survey papers by
  \cite{Bickel06} and \cite{Fan10} for the references. These penalties,
  as a convex surrogate (say, $\ell_1$-norm) or a nonconvex approximation
  (say, the bridge penalty) to the zero-norm, essentially try to capture
  the performance of the zero-norm, first used in the best subsect selection
  by \cite{Breiman96}. The sparse least squares regression approach is useful,
  but it only focuses on the central tendency of the conditional distribution.
  It is known that a certain covariate may not have significant influence on
  the mean value of the response but may have a strong effect on the upper quantile
  of the conditional distribution due to the heterogeneity of data.
  It is likely that a covariate has different effects at different segments
  of the conditional distribution. As illustrated by \cite{Koenker78},
  for non-Gaussian error distributions, the least squares regression is substantially
  out-performed by the quantile regression (QR).

  \medskip

  Inspired by this, many researchers recently have considered the QR
  introduced by \cite{Koenker78} for high-dimensional data analysis,
  owing to its robustness to outliers and its ability to offer unique
  insights into the relation between the response variable and the covariates;
  see, e.g., \cite{Wu09,Belloni11,Wang12,Wang13,Fan14a,Fan14b}. \cite{Belloni11}
  focused on the theory of the $\ell_1$-penalized QR and showed that
  this estimator is consistent at the near-oracle rate and provided
  the conditions under which the selected model includes the true model;
  \cite{Wang13} studied the $\ell_1$-penalized least absolute derivation (LAD) regression
  and verified that the estimator has near oracle performance with a high probability;
  and \cite{Fan14a} studied the weighted $\ell_1$-penalized QR and
  established the model selection oracle property and the asymptotic normality
  for this estimator. For nonconvex penalty-type QRs, \cite{Wu09} under
  mild conditions achieved the asymptotic oracle property of the SCAD and
  adaptive-Lasso penalized QRs, and \cite{Wang12} showed that
  with probability approaching one, the oracle estimator is a local optimal
  solution to the SCAD or MCP penalized QRs of ultra-high dimensionality.
  We notice that the above results are all established for the asymptotic case $n\to\infty$.

  \medskip

  Besides the above theoretical works, there are some works concerned with
  the computation of (weighted) $\ell_1$-penalized QR estimators which,
  compared to the (weighted) $\ell_1$-least-squares estimator, requires more
  sophisticated algorithms due to the piecewise linearity of the check loss function.
  Although the $\ell_1$-penalized QR model can be transformed into a linear program (LP)
  by introducing additional variables and one may use the interior point method
  (IPM) softwares such as SeDuMi in \cite{Sturm99} to solve it, this is limited to
  the small or medium scale case; see Figure \ref{fig1}-\ref{fig2} in Section \ref{sec5}.
  Inspired by this, \cite{Wu08} proposed a greedy coordinate descent algorithm for
  the $\ell_1$-penalized LAD regression, \cite{Yi17} proposed a semismooth Newton
  coordinate descent algorithm for the elastic-net penalized QR, and \cite{Gu18}
  recently developed a semi-proximal alternating direction method of multipliers (sPADMM)
  and a combined version of ADMM and coordinate descent method (which is actually an inexact
  ADMM) for solving the weighted $\ell_1$-penalized QR. In addition, for nonconvex penalized QRs,
  \cite{Peng15} developed an iterative coordinate descent algorithm
  and established the convergence of any subsequence to a stationary point,
  and \cite{Fan14b} provided a systematic study for folded concave
  penalized regressions, including the SCAD and MCP penalized QRs as special cases,
  and showed that with high probability the oracle estimator can be obtained within
  two iterations of the local linear approximation (LLA) approach proposed by \cite{Zou08}.
  We find that \cite{Peng15} and \cite{Fan14b} did not establish the error bound
  of the iterates to the true solution.

  \medskip

  This work is interested in the computation of the high-dimensional zero-norm
  penalized QR estimator, a global minimizer of the zero-norm regularized
  check loss. To seek a high-quality approximation to this estimator,
  we reformulate this NP-hard problem as a mathematical program with an equilibrium
  constraint (MPEC), and obtain an equivalent augmented Lipschitz optimization problem
  from the global exact penalty of the MPEC. This augmented problem not only has a favorable
  coupled structure but also implies an equivalent DC (difference of convex)
  surrogate for the zero-norm regularized check loss minimization; see Section \ref{sec2}.
  By solving the augmented Lipschitz problem in an alternating way,
  we propose in Section \ref{sec3} an MSCRA to compute a desirable surrogate
  for the zero-norm penalized QR estimator. Similar to the LLA method owing to
  \cite{Zou08}, the MSCRA solves in each step a weighted $\ell_1$-regularized
  check loss minimization, but the subproblems are allowed to be solved inexactly.
  Under a mild restricted strong convexity condition, we provide its theoretical
  guarantee in Section \ref{sec4} by establishing the error bound
  of each iterate to the true estimator and the rate of linear convergence
  in a statistical sense.

  \medskip

  Motivated by the recent work \cite{Tang19}, we also develop a proximal
  dual semismooth Newton method (PDSN) in Section \ref{sec5} for solving
  the subproblems involved in the MSCRA. Different from the semismooth Newton method
  by \cite{Yi17}, this is a proximal point algorithm (PPA) with the subproblems
  solved by applying the semismooth Newton method to their duals,
  rather than to a smooth approximation to the elastic-net penalized check
  loss minimization problem. Numerical comparisons are made on some synthetic
  and real data for MSCRA\_PPA, MSCRA\_IPM and MSCRA\_ADMM, which are the MSCRA
  with the subproblems solved by PDSN, SeDuMi in \cite{Sturm99} and
  semi-proximal ADMM in \cite{Gu18}, respectively. We find that
  MSCRA\_IPM and MSCRA\_ADMM have very similar performance, while MSCRA\_PPA
  not only has a comparable estimation performance with the two methods
  but also requires only one-fifteenth of the CPU time required by
  MSCRA\_ADMM and MSCRA\_IPM.

  \medskip

  Throughout this paper, $I$ and $e$ denote an identity matrix
  and a vector of all ones, whose dimensions are known from the context.
  For an $x\in\mathbb{R}^p$, write $|x|:=(|x_1|,\ldots,|x_p|)^{\mathbb{T}}$
  and ${\rm sign}(x):=({\rm sign}(x_1),\ldots,{\rm sign}(x_p))^{\mathbb{T}}$,
  and denote by $\|x\|_1, \|x\|$ and $\|x\|_\infty$ the $l_1$-norm,
  $l_2$-norm and $l_\infty$-norm of $x$, respectively. For a matrix $A\in\mathbb{R}^{n\times p}$,
  $\|A\|,\|A\|_{\max}$ and $\|A\|_1$ respectively denote the spectral norm,
  element-wise maximum norm, and maximum column sum norm of $A$.
  For a set $S$, $\mathbb{I}_{S}$ means the characteristic function on $S$,
  i.e., $\mathbb{I}_{S}(z)=1$ if $z\in S$, otherwise $\mathbb{I}_{S}(z)=0$.
  For given $a,b\in\mathbb{R}^p$ with $a_i\le b_i$ for $i=1,\ldots,p$,
  $[a,b]$ means the box set. For an extended real-valued function
  $f\!:\mathbb{R}^p\to(-\infty,+\infty]$,
  write ${\rm dom}\,f:=\{x\in\mathbb{R}^p\ |\ f(x)<\infty\}$,
  and denote $\mathcal{P}_{\gamma}f$ and $e_{\gamma}f$ for a given $\gamma>0$
  by the proximal mapping and Moreau envelope of $f$, defined as
  $\mathcal{P}_{\gamma}f(x):=\mathop{\arg\min}_{z\in\mathbb{R}^p}\big\{f(z)+\frac{1}{2\gamma}\|z-x\|^2\big\}$
  and $e_{\gamma}f(x):=\min_{z\in\mathbb{R}^p}\big\{f(z)+\frac{1}{2\gamma}\|z-x\|^2\big\}$.
  In the sequel, we write $\mathcal{P}\!f$ for $\mathcal{P}_{1}f$.
  When $f$ is convex, $\mathcal{P}_{\gamma}f\!:\mathbb{R}^p\to\mathbb{R}^p$
  is a Lipschitz mapping with modulus $1$, and $e_{\gamma}f$ is
  a smooth convex function with $\nabla e_{\gamma}f(x)=\gamma^{-1}(x-\mathcal{P}_{\gamma}f(x))$.

 \section{Zero-norm penalized quantile regression and equivalent difference of convex model}\label{sec2}

 Quantile regression is a popular method for studying the influence of
 a set of covariates on the conditional distribution of a response variable,
 and has been widely used to handle heteroscedasticity; see \cite{Koenker82}
 and \cite{Wang12}. For a univariate response ${\bf Y}$ and
 a vector of covariates ${\bf X}\in\mathbb{R}^p$,
 the conditional cumulative distribution function of ${\bf Y}$ is defined as
 $F_{\bf Y}(t|x):={\rm Pr}({\bf Y}\leq t\ |\ {\bf X}=x)$, and the $\tau$th conditional
 quantile of ${\bf Y}$ is given by
 \(
   Q_{\bf Y}(\tau|x):=\inf\big\{t\!: F_{\bf Y}(t|x)\geq \tau\big\}.
 \)
 Let $X\!=[x_1\ \cdots\ x_n]^{\mathbb{T}}$ be an $n\times p$ design matrix
 on ${\bf X}$. Consider the linear quantile regression
 \begin{equation}\label{QR-model}
   y = X\beta^* +\varepsilon
 \end{equation}
 where $y=(y_1,\ldots,y_n)^{\mathbb{T}}\in\mathbb{R}^n$ is the response vector, $\varepsilon=(\varepsilon_1,\ldots,\varepsilon_n)^{\mathbb{T}}$
 is the noise vector whose components are independently distributed and
 satisfy ${\rm Pr}(\varepsilon_i\le 0|x_i)=\tau$ for some known constant $\tau\in(0,1)$,
 and $\beta^*\in\mathbb{R}^p$ is the true but unknown coefficient vector.
 This quantile regression model actually assumes that
 $Q_{\bf Y}(\tau|x_i) = x_i^{\mathbb{T}}\beta^*$ for $i=1,\ldots,n$.
 We are interested in the high-dimensional case where $p>n$ and the sparse model
 in the sense that only $s^*(\ll p)$ components of the unknown true $\beta^*$ are nonzero.

 \medskip

 For $\tau\in\!(0,1)$, let $f_{\tau}\!:\mathbb{R}^n\to\mathbb{R}$
 be the check loss function of \eqref{QR-model}, i.e.,
 \begin{equation}\label{ftau}
  f_{\tau}(z):=n^{-1}{\textstyle\sum_{i=1}^n}\theta_{\tau}(z_i)
  \ \ {\rm with}\ \ \theta_{\tau}(u):=(\tau-\mathbb{I}_{\{u\le 0\}})u
 \end{equation}
 which was first introduced by \cite{Koenker78}.
 To estimate the unknown true $\beta^*$ in \eqref{QR-model}, we consider
 the zero-norm regularized problem
 \begin{equation}\label{prob}
   \widehat{\beta}(\tau)\in\mathop{\arg\min}_{\beta\in\mathbb{R}^p}
   \Big\{\nu f_{\tau}(y-\!X\beta)+\|\beta\|_0\Big\}
  \end{equation}
  where $\nu>0$ is the regularization parameter, and $\|\beta\|_0$ denotes the zero-norm
  of $\beta$ (i.e., the number of nonzero entries of $\beta$). By the expression of $f_{\tau}$,
  $f_{\tau}$ is nonnegative and coercive (i.e., $f_{\tau}(\beta^k)\to +\infty$
  whenever $\|\beta^k\|\to\infty$). By Lemma 3 in Appendix A,
  the estimator $\widehat{\beta}(\tau)$ is well defined.
  Since $\widehat{\beta}(\tau)$ depends on $\tau$, there is a great possibility
  for model \eqref{prob} to monitor different ``locations'' of the conditional distribution,
  and then the heteroscedasticity of the data, when existing,
  can be inspected by solving \eqref{prob} with different $\tau\in(0,1)$.
  For the simplicity, in the sequel we use $\widehat{\beta}$ to
  replace $\widehat{\beta}(\tau)$, and for a given $\tau\in(0,1)$, write
  $\underline{\tau}:=\min(\tau,1\!-\!\tau)$ and $\overline{\tau}:=\max(\tau,1\!-\!\tau)$.

  \medskip

  Due to the combination of the zero-norm, the computation of $\widehat{\beta}$
  is NP-hard. To design an algorithm in the next section for seeking
  a high-quality approximation to $\widehat{\beta}$, we next derive an equivalent
  augmented Lipschitz optimization problem from a primal-dual viewpoint,
  and to demonstrate that such a mechanism provides a unified way to yield
  equivalent DC surrogates for the zero-norm regularized problem \eqref{prob},
  we introduce a  family of proper lsc convex functions on $\mathbb{R}$,
  denoted by $\mathscr{L}$, satisfying the conditions:
  \begin{equation}\label{phi-assump}
   {\rm int}({\rm dom}\,\phi)\supseteq[0,1],\
   t^*\!:=\mathop{\arg\min}_{0\le t\le 1}\phi(t),\ \phi(t^*)=0
   \ \ {\rm and}\ \ \phi(1)=1.
  \end{equation}
  With a $\phi\in\!\mathscr{L}$, clearly, the zero-norm $\|z\|_0$ is the optimal value function of
  \begin{equation*}
   \min_{w\in\mathbb{R}^p}\Big\{{\textstyle\sum_{i=1}^p}\phi(w_i)\quad\mbox{s.t.}
   \ \ \langle e-w,|z|\rangle=0,\,0\le w\leq e\Big\}.
  \end{equation*}
  This characterization of zero-norm shows that model \eqref{prob} is equivalent to
 \begin{equation}\label{Eprob}\small
  \min_{\beta\in\mathbb{R}^p,w\in\mathbb{R}^p}\bigg\{\nu f_{\tau}(y-\!X\beta)
   +\sum_{i=1}^p\phi(w_i)\quad\mbox{s.t.}\ \ \langle e\!-w,|\beta|\rangle=0,\,0\le w\leq e\bigg\}
 \end{equation}
 in the following sense: if $\overline{\beta}$ is globally optimal to \eqref{prob},
 then $(\overline{\beta}\!,{\rm sign}(|\overline{\beta}|))$ is a global optimal
 solution of problem \eqref{Eprob}, and conversely, if $(\overline{\beta},\overline{w})$
 is a global optimal solution of \eqref{Eprob}, then $\overline{\beta}$ is
 globally optimal to \eqref{prob}. Problem \eqref{Eprob} is a mathematical
 program with an equilibrium constraint $e-w\ge 0,|\beta|\ge 0$, $\langle e-w,|\beta|\rangle=0$
 (abbreviated as MPEC). The equivalence between \eqref{prob} and \eqref{Eprob}
 shows that the difficulty of model \eqref{prob} arises from
 the hidden equilibrium constraint. It is well known that
 the handling of nonconvex constraints is much harder than that of nonconvex
 objective functions. Then it is natural to consider the penalized version
 of problem \eqref{Eprob}
 \begin{equation}\label{Eprob-penalty}
  \min_{\beta\in\mathbb{R}^p,w\in[0,e]}\Big\{\nu f_{\tau}(y-\!X\beta)
     +\big[\textstyle{\sum_{i=1}^p}\phi(w_i)+\rho\langle e-w,|\beta|\rangle\big]\Big\}
 \end{equation}
 where $\rho>0$ is the penalty parameter. Since $\beta\mapsto\!f_{\tau}(y-\!X\beta)$
 is Lipschitz continuous, the following conclusion holds by Section 3.2 of \cite{LiuBiPan18}.
 \begin{theorem}\label{theorem-epenalty}
  The problem \eqref{Eprob-penalty} associated to each
  $\rho>\overline{\rho}:=\frac{ \phi_{-}'(1)(1-t^*)\overline{\tau}\nu\|X\|}{1-t_0}$
  has
  the same global optimal solution set as the MPEC \eqref{Eprob} does,
  where $t^0$ is the minimum element in $[t^*,1)$ such that
  $\frac{1}{1-t^*}\in\partial\phi(t_0)$.
 \end{theorem}

 Theorem \ref{theorem-epenalty} states that problem \eqref{Eprob-penalty}
 is a global exact penalty of \eqref{Eprob} in the sense that there is
 a threshold $\overline{\rho}>0$ such that the former associated
 to every $\rho>\overline{\rho}$ has the same global optimal solution set
 as the latter does. Together with the equivalence between \eqref{prob}
 and \eqref{Eprob}, model \eqref{prob} is equivalent to problem
 \eqref{Eprob-penalty}. Notice that the objective function of \eqref{Eprob-penalty}
 is globally Lipschitz continuous over its feasible set and its nonconvexity is owing to
 the coupled term $\langle e\!-\!w,|\beta|\rangle$ rather than the combination.
 So, problem \eqref{Eprob-penalty} provides an equivalent augmented
 Lipschitz reformulation for the zero-norm problem \eqref{prob}.
 In fact, problem \eqref{Eprob-penalty} associated to every $\rho>\overline{\rho}$
 implies an equivalent DC surrogate for \eqref{prob}. To illustrate this, let
 $\psi(t)=\phi(t)$ if $t\in[0,1]$ and otherwise $\phi(t)=+\infty$.
  Then, with the conjugate $\psi^*(s):=\sup_{t\in\mathbb{R}}\{st-\psi(t)\}$
  of $\psi$, one may check that \eqref{Eprob-penalty} is equivalent to
 \begin{equation}\label{Eprob-conj}
   \min_{\beta\in\mathbb{R}^p}\Big\{\Theta_{\nu,\rho}(\beta)
   :=f_{\tau}(y-\!X\beta)+\nu^{-1}{\textstyle\sum_{i=1}^p}\big[\rho|\beta_i|-\psi^*(\rho|\beta_i|)\big]\Big\}.
  \end{equation}
 Since $\psi^*$ is a nondecreasing finite convex function on $\mathbb{R}$,
 the function $s\mapsto\psi^*(\rho|s|)$ is convex, and problem \eqref{Eprob-conj}
 is a DC program. To sum up the above discussions, problem \eqref{Eprob-conj}
 associated to every $\rho>\overline{\rho}$ provides
 an equivalent DC surrogate for \eqref{prob}. Moreover,
 $H_{\rho}(\beta):=\sum_{i=1}^ph_{\rho}(\beta_i)$
 with $h_{\rho}(t):=\rho|t|-\psi^*(\rho|t|)$ for $t\in\mathbb{R}$
 is a DC surrogate for the zero-norm. To close this section,
 we present some examples of $\phi\in\mathscr{L}$.
 \begin{example}\label{examplephi1}
 Let $\phi(t)=t$ for $t\in\mathbb{R}$. After a simple computation, we have
 \[
    \psi^*(s)=\left\{\begin{array}{cl}
                 0   & {\rm if}\ s\leq1,\\
                 s-1 & {\rm if}\ s>1
             \end{array}\right.
~~~{\rm and}~~~~
    h_{\rho}(t)=\left\{\begin{array}{cl}
                 \rho|t|   & {\rm if}\ |t|\leq\frac{1}{\rho},\\
                 1         & {\rm if}\ |t|>\frac{1}{\rho}.
             \end{array}\right.
  \]
  It is immediate to see that the function $\nu^{-1}h_{\rho}(t)$ will reduce to
  the capped $\ell_1$-function $t\mapsto\lambda\min(|t|,\alpha)$ in \cite{ZhangT10} with
  $\nu=\rho/\lambda$ and $\rho=\alpha^{-1}$.
 \end{example}
 \begin{example}\label{examplephi2}
  Let $\phi(t):=\frac{a-1}{a+1}t^2+\frac{2}{a+1}t\ (a>1)$ for $t\in \mathbb{R}$.
  One can calculate
  \begin{align}\label{psi-star}
  \psi^*(s)&=\left\{\begin{array}{cl}
                      0 & \textrm{if}\ s\leq \frac{2}{a+1},\\
                      \frac{((a+1)s-2)^2}{4(a^2-1)} & \textrm{if}\ \frac{2}{a+1}<s\leq \frac{2a}{a+1},\\
                      s-1 & \textrm{if}\ s>\frac{2a}{a+1};
                \end{array}\right.\\
  h_{\rho}(t)&=\left\{\begin{array}{cl}
                      \rho|t| & \textrm{if}\ |t|\leq \frac{2}{(a+1)\rho},\\
                      \rho|t|-\frac{((a+1)\rho|t|-2)^2}{4(a^2-1)} & \textrm{if}\ \frac{2}{(a+1)\rho}<|t|\leq \frac{2a}{(a+1)\rho},\\
                      1 & \textrm{if}\ |t|>\frac{2a}{(a+1)\rho}.
                \end{array}\right.\nonumber
  \end{align}
  It is not hard to check that $\nu^{-1}h_{\rho}(t)$ will reduces to
  the SCAD function $\rho_{\lambda}(t)$ in \cite{Fan01}
  when $\nu=\frac{2}{(a+1)\lambda^2}$ and $\rho=\frac{2}{(a+1)\lambda}$.
  \end{example}
 \begin{example}\label{examplephi3}
  Let $\phi(t):=\frac{a^2}{4}t^2-\frac{a^2}{2}t+at+\frac{(a-2)^2}{4}\ (a>2)$ for $t\in\mathbb{R}$.
  We have
  \begin{align*}
  \psi^*(s)&=\left\{\begin{array}{cl}
                      -\frac{(a-2)^2}{4} & \textrm{if}\ s\leq  a-a^2/2,\\
                      \frac{1}{a^2}(\frac{a(a-2)}{2}+s)^2-\frac{(a-2)^2}{4}& \textrm{if}\  a-a^2/2 <s\leq a,\\
                      s-1 & \textrm{if}\ s>a;
                \end{array}\right.\\
  h_{\rho}(t)&=\left\{\begin{array}{cl}
                      \rho|t|-\frac{1}{a^2}(\frac{a(a-2)}{2}+\rho|t|)^2+\frac{(a-2)^2}{4}& \textrm{if}\  |t|\leq {a}/{\rho},\\
                      1 & \textrm{if}\ |t|>{a}/{\rho}.
                \end{array}\right.
  \end{align*}
  The $\nu^{-1}h_{\rho}(t)$ will reduce to the MCP in \cite{Zhang10}
  if $\nu=\frac{2}{ a\lambda^2 },\rho=\frac{1}{\lambda}$.
  \end{example}
 \section{Multi-stage convex relaxation approach }\label{sec3}

  From the last section, to compute the estimator $\widehat{\beta}$,
  we only need to solve a single penalty problem \eqref{Eprob-penalty}
  that is much easier than the zero-norm problem \eqref{prob}
  because its nonconvexity only arises from the coupled term $\langle w,|\beta|\rangle$.
  Observe that \eqref{Eprob-penalty} becomes a convex program when
  either of $w$ and $\beta$ is fixed. So, we solve it in an alternating way
  and propose the following multi-stage convex relaxation approach (MSCRA)
  with $\phi$ in Example \ref{examplephi2}.
 \begin{algorithm}[h]
 \caption{\label{Alg}{\bf(MSCRA for computing $\widehat{\beta}$)}}
 \textbf{Initialization:} Choose $\tau\in(0,1), \nu>0,\rho_0=1,w^0\!\in[0,\frac{1}{2}e]$.
                          Set $\lambda=\frac{\rho_0}{\nu}$.\\
 \textbf{for} $k=1,2,\ldots.$
 \vspace{-0.3cm}
 \begin{enumerate}
  \item  Seek an inexact solution to the weighted $\ell_1$-regularized problem
         \vspace{-0.3cm}
         \begin{equation}\label{subprob-betak}
         \vspace{-0.3cm}
          \beta^k\approx\mathop{\arg\min}_{\beta\in\mathbb{R}^p}
           \Big\{f_{\tau}(y-\!X\beta)+\lambda\,{\textstyle\sum_{i=1}^p}(1\!-\!w_i^{k-1})|\beta_{i}|\Big\}.
         \end{equation}
  \item When $k=1$, select a suitable $\rho_1\ge\rho_0$ in terms of $\|\beta^1\|_\infty$.
        If $k=2,3$, select $\rho_k$ such that $\rho_k\ge\rho_{k-1}$;
        otherwise, set $\rho_k=\rho_{k-1}$.
  \item For $i=1,2,\ldots,p$, compute the following minimization problem
        \vspace{-0.3cm}
        \begin{equation}\label{subprob-wk}
        \vspace{-0.3cm}
              w_i^k=\mathop{\rm arg\min}_{0\le w_i\le 1}
              \left\{\phi(w_i)-\rho_k w_i|\beta^k_{i}|\right\}.
        \end{equation}
  \end{enumerate}
  \vspace{-0.3cm}
  \textbf{end for}
\end{algorithm}
 \begin{remark}\label{remark-Alg}
 {\bf(i)} Step 1 of Algorithm \ref{Alg} is solving problem \eqref{Eprob-penalty}
  with $w$ fixed to be $w^{k-1}$, while Step 3 is solving this problem
  with $\beta$ fixed to be $\beta^k$; that is, Algorithm \ref{Alg} is solving
  the nonconvex penalty problem \eqref{Eprob-penalty} in an alternating way.
  In the first stage, since there is no any information
  on estimating the nonzero entries of $\beta^*$, it is reasonable to impose
  an unbiased weight on each component of $\beta$. Motivated by this, we restrict
  the initial $w^0$ in $[0,0.5e]$, a subset of the feasible set of $w$.
  When $w^0=0$, the first stage is precisely the minimization of
  the $\ell_1$-penalized check loss function. Although the threshold $\overline{\rho}$
  is known when the parameter $\nu$ in \eqref{prob} is given, we select
  a varying $\rho$ for \eqref{subprob-wk} since it is just a relaxation
  of \eqref{Eprob-penalty}.

  \noindent
  {\bf(ii)} By the optimality condition of \eqref{subprob-wk},
  $\rho_k|\beta_i^k|\in\partial\psi(w_i^k)$ for each $i$, which by Theorem 23.5
  in \cite{Roc70} and \eqref{psi-star} is equivalent to saying
  \begin{equation}\label{wik}
   w_i^{k}=\min\Big[1,\max\Big(0,\frac{(a+1)\rho_k|\beta_i^{k}|-2}{2(a-1)}\Big)\Big]
   \ \ {\rm for}\ i=1,\ldots,p.
  \end{equation}
  Clearly, when $\rho_k|\beta_i^{k}|$ is close to $0$, $(1\!-\!w_i^k)$ in \eqref{wik}
  may not equal $1$ though close to $1$; when $\rho_k|\beta_i^{k}|$ is very larger,
  $(1\!-\!w_i^k)$ in \eqref{wik} may not equal $0$ though close to $0$.
  To achieve a high-quality solution with Algorithm \ref{Alg}, the last term of \eqref{subprob-betak}
  implies that a smaller $(1\!-\!w_i^{k-1})$ but not $0$ is expected for those larger $|\beta_i|$,
  and a larger $(1\!-\!w_i^{k-1})$ instead of $1$ is expected for those smaller $|\beta_i|$.
  Thus, the function $\phi$ in Example \ref{examplephi2} is desirable especially
  for those problems whose solutions have small nonzero entries.
  The weight $w^k$ associated to the function $\phi$ in Example \ref{examplephi3}
  has a similar performance, but the weight $w^k$ associated to the function
  $\phi$ in Example \ref{examplephi1} is different since $w_i^k=0$ if
  $\rho_k|\beta_i^{k}|<1$, $w_i^k=1$ if $\rho_k|\beta_i^{k}|>1$, otherwise
  $w_i^k\in[0,1]$.

 \noindent
 {\bf(iii)} Algorithm \ref{Alg} is actually an inexact majorization-minimization (MM) method
 (see \cite{Lange00}) for solving the equivalent DC surrogate \eqref{Eprob-conj} with
 a special starting point. Indeed, for a given $\beta'\in\mathbb{R}^p$,
  the convexity and smoothness of $\psi^*$ implies that
  with $w_i=(\psi^*)'(\rho|\beta_i'|)$ for $i=1,\ldots,p$,
  \begin{equation}\label{ineq-psistar}
   \sum_{i=1}^p\psi^*(\rho|\beta_i|)\ge\sum_{i=1}^p\psi^*(\rho|\beta_i'|)
   +\rho\langle w,|\beta|-|\beta'|\rangle\quad\ \forall \beta\in\mathbb{R}^p.
  \end{equation}
  Notice that each $w_i\in[0,1]$ by the expression of $\psi^*$.
  Hence, the function
  \[
    f_{\tau}(y-\!X\beta)+\lambda\big\|(e-\!w^{k-1})\circ \beta\big\|_1
    -\lambda\big[\sum_{i=1}^p\psi^*(\rho|\beta_i^{k-1}|)+\rho\langle w^{k-1},|\beta^{k-1}|\rangle\big]
  \]
  is a majorization of $\Theta_{\lambda,\rho}$ at $\beta^{k-1}$ and
  the subproblem \eqref{subprob-betak} is the inexact minimization of
  this majorization function. Also, for any given $\rho_0>0$,
  when $\|\beta^0\|_{\infty}\le\frac{2}{(a+1)\rho_0}$,
  we have $w_i^0=(\psi^*)'(\rho_0|\beta_i^0|)=0$ by \eqref{psi-star}.
  Thus, the first stage of Algorithm \ref{Alg} with $w^0=0$ is precisely
  the inexact MM method for \eqref{Eprob-conj} with $\beta^0$ satisfying
  $\|\beta^0\|_{\infty}\le\frac{2}{(a+1)\rho_0}$. In addition, Algorithm \ref{Alg}
  can be regarded as an inexact inversion of the LLA method proposed
  by \cite{Zou08} for \eqref{Eprob-conj}, but it is different from
  the DC algorithm by \cite{Wu09} since
  the latter depends on the majorization of
  $\beta\mapsto {\textstyle\sum_{i=1}^p}\psi^*(\rho|\beta_i|)$ at $\beta^k$
  and the obtained approximation is lack of symmetry.

 \noindent
 {\bf(iv)} Considering that practical computation always involves deviation,
  we allow the problem in \eqref{subprob-betak} to be solved inexactly with
  the accuracy measured in the following way:
  $\exists\delta^{k}\in\mathbb{R}^p$ and $r_k\ge 0$ with
  $\|\delta^{k}\|\le r_k$ such that
  \begin{align}\label{inexact-inclusion}
   \delta^{k}&\in\partial\big[f_{\tau}(y-\!X\beta)+\lambda\|(e-\!w^{k-1})\circ \beta\|_1\big]_{\beta=\beta^{k}}\nonumber\\
   &=-X^{\mathbb{T}}\partial\!f_{\tau}(y\!-\!X\beta^k)+\lambda\big[(1\!-\!w_1^{k-1})\partial|\beta_1^k|
     \times\cdots\times(1\!-\!w_p^{k-1})\partial|\beta_p^k|\big]
  \end{align}
  where the equality is by Theorem 23.8 in \cite{Roc70}.
  Notice that the first-order optimality conditions of \eqref{Eprob-penalty}
  take the following form
  \begin{align*}
   u\in\partial\!f_{\tau}(z);\ \rho|\beta_i|\in\partial\psi(w_i)\ {\rm for}\ i=1,\ldots,p;\
   y-\!X\beta-z=0;\\
   X^{\mathbb{T}}u\in\lambda\big[(1\!-\!w_1)\partial|\beta_1|
     \times\cdots\times(1\!-\!w_p)\partial|\beta_p|\big],\qquad\quad
  \end{align*}
  where $u\in\mathbb{R}^n$ is the Lagrange multiplier associated to
  $y-X\beta-z=0$. By Step 2 of Algorithm \ref{Alg},
  $\rho_k|\beta^k|\in\partial\psi(w_1^k)\times\cdots\times\partial\psi(w_p^k)$.
  In view of this, we measure the KKT residual of \eqref{Eprob-penalty}
  associated to $\rho_k$ at $(\beta^k,z^k,u^k)$ by
  \begin{equation}\label{errk}
   {\bf Err}_k:=\frac{\sqrt{\|\Delta_1\|^2+\|\Delta_2^k\|^2+\|y-\!X\beta^k-\!z^k\|^2}}{1+\|y\|}
   \le{\rm tol}
  \end{equation}
  where $\Delta_1^k:=z^k-\mathcal{P}\!f_{\tau}(z^k+u^k)$
  and $\Delta_2^k:=X^{\mathbb{T}}u^k-\mathcal{P}h_{k}(X^{\mathbb{T}}u^k+\beta^k)$ with
  \begin{equation}\label{hkfun}
    h_{k}(\beta):=\|\lambda(e\!-\!w^k)\circ\beta\|_1
    \ \ {\rm for}\ \beta\in\mathbb{R}^p.
  \end{equation}
  \end{remark}

 \section{Theoretical guarantees of Algorithm \ref{Alg}}\label{sec4}

We denote by $S^{*}$ the support of the true vector $\beta^{*}$,
 and define the set
 \[
  \mathcal {C}(S^{*}):=\bigcup_{S^{*}\subset S,|S|\le 1.5s^*}\!\Big\{\beta\in \mathbb{R}^p\!:
  \|\beta_{S^c}\|_1\le 3\|\beta_{S}\|_1\Big\}.
 \]
 The matrix $X$ is said to have the $\kappa$-restricted strong
 convexity on $\mathcal {C}(S^{*})$ if
 \begin{equation}\label{RSC}
  \kappa>0\ \ {\rm and}\ \
  \frac{1}{2n}\|X\Delta\beta\|^2\ge\kappa\|\Delta\beta\|^2\quad{\rm for\ all}\ \Delta\beta\in\mathcal{C}(S^*).
 \end{equation}
 The RSC is equivalent to the restricted eigenvalue condition of the Gram matrix
 $\frac{1}{2n}X^{\mathbb{T}}X$ due to \cite{van09} and \cite{Bickel09}.
 Notice that $\mathcal{C}(S^*)\supseteq
 \big\{\beta\in\mathbb{R}^p\!: \|\beta_{(S^*)^c}\|_1\le 3\|\beta_{S^*}\|_1\big\}$.
 This RSC is a little stronger than the one used by \cite{Negahban12}
 for the $\ell_1$-regularized smooth loss minimization. In this section,
 we shall provide the deterministic theoretical guarantees for Algorithm \ref{Alg}
 under this RSC, including the error bound of the iterate $\beta^k$ to the true
 $\beta^{*}$ and the decrease analysis of the error sequence.
 The proofs are all included in Appendix B. We need the following assumption
 on the optimality tolerance $r_k$ of $\beta^k$:
 \begin{assumption}\label{assump1}
  There exists $\epsilon>0$ such that for each $k\in\mathbb{N}$,
  $r_k\le\epsilon$.
 \end{assumption}

 First, by Lemma \ref{betak-lemma1} in Appendix B, we have the following error bound.
 \begin{theorem}\label{error-bound1}
  Suppose that Assumption \ref{assump1} holds, that $X$ has the $\kappa$-RSC over
  $\mathcal{C}(S^*)$, and that the noise vector $\varepsilon$ is nonzero. If $\rho_3$ and $\lambda$
  are chosen such that $\rho_3\le\frac{8}{9\sqrt{3}c\overline{\tau}\lambda\|\varepsilon\|_\infty}$
  and
  \(
  \lambda\in\Big[\frac{16\overline{\tau}\|X\|_1}{n}+8\epsilon ,
  \frac{ \underline{\tau}^2\kappa-c^{-1}-3\overline{\tau}\|X\|_{\rm max}(2n^{-1}\overline{\tau} \|X\|_1+\epsilon) s^*  }
  {3\overline{\tau}\|X\|_{\rm max} s^* }\Big]
  \)
  for some constant $$c\ge\frac{1}{\underline{\tau}^2\kappa-27 \overline{\tau}\|X\|_{\rm max}
  (2n^{-1}\overline{\tau}\|X\|_1+\epsilon) s^* },$$ then for every $k\in\mathbb{N}$
  \[
    \|\beta^{k}-\beta^*\|\le\frac{9c\overline{\tau}\lambda\sqrt{1.5s^*}}{8}\|\varepsilon\|_\infty.
  \]
 \end{theorem}
 \begin{remark}
 {\bf(i)} For the $\ell_1$-regularized least squares smooth loss estimator
 $$
   \beta^{\rm LS}\in\mathop{\arg\min}_{\beta\in\mathbb{R}^p}
   \Big\{\frac{1}{2n}\|y-\!X\beta\|^2+\lambda_n\|\beta\|_1\Big\},
 $$
 the error bound $\|\beta^{\rm LS}-\beta^*\|=O(\sigma\sqrt{s^*\log p/n})$
 was obtained in Corollary 2 of \cite{Negahban12} by taking
 $\lambda_n=\sqrt{\log p/n}$, where $\sigma>0$ represents the variance
 of the noise. By comparing with this error bound, the error bound
 in Theorem \ref{error-bound1} involves the infinite norm $\|\varepsilon\|_\infty$
 of noise $\varepsilon$ rather than its variance, and moreover,
 it still has the same order $O(\sqrt{s^*\log p/n})$ when the parameter
 $\lambda=O(1)$ in our model is rescaled to be $\lambda_n$.

 \noindent
 {\bf(ii)} For the following $\ell_1$-regularized square-root nonsmooth loss estimator
 $$
   \beta^{\rm sr}\in\mathop{\arg\min}_{\beta\in\mathbb{R}^p}
   \Big\{\frac{1}{\sqrt{n}}\|y-\!X\beta\|+\frac{\lambda'}{n}\|\beta\|_1\Big\},
 $$
 the error bound $\|\beta^{\rm sr}\!-\!\beta^*\|=O\big(\frac{\sigma\sqrt{s^*}\lambda'\varpi}{n}\big)$
 with $\varpi\ge\frac{1}{\sqrt{n}}\|\varepsilon\|$ was achieved in Theorem 1
 of \cite{Belloni110} by setting $\lambda'=O(n)$.
 By considering that $f_{\tau}(y-X\beta)=O(\sqrt{n}\|y-X\beta\|)$,
 the parameter $\lambda$ in our model corresponds to $\lambda'/n$.
 Thus, the error bound in Theorem \ref{error-bound1} corresponds to
 $O(\frac{\sqrt{s^*}\lambda'\|\varepsilon\|_\infty}{n})$, which has the same order
 as $O\big(\frac{\sigma\sqrt{s^*}\lambda'\varpi}{n}\big)$ since
 $\|\varepsilon\|_\infty=O(\frac{1}{\sqrt{n}}\|\varepsilon\|)$.

 \noindent
 {\bf(iii)} To ensure that the constant $c>0$ exists, the constant $\kappa$
 needs to satisfy $\kappa>\frac{54\overline{\tau}^2 s^*\|X\|_{\rm max}\|X\|_1}
 {n\underline{\tau}^2}$ and the inexact accuracy $\epsilon$ of $\beta^{k}$ needs to satisfy
 $$0\le\epsilon<\frac{n\underline{\tau}^2\kappa-54\overline{\tau}^2s^*\|X\|_{\rm max}\|X\|_1}{27n\overline{\tau} s^* }.$$
 Since $\|X\|_1=O(n)$, it is necessary to solve the subproblem \eqref{subprob-betak}
 with a very small inexact accuracy $\epsilon$.
 \end{remark}

 Theorem \ref{error-bound1} establishes an error bound for every iterate $\beta^k$,
 but it does not tell us if the error bound of the current $\beta^k$ is
 better than that of the previous $\beta^{k-1}$. In order to seek the answer,
 we study the decrease of the error bound sequence by bounding
 $\max_{i\in S^*}(1-w_i^k)$. For this purpose, write $F^0:=S^*$ and
 $\Lambda^0:=\{i\!:|\beta_i^*|\le \frac{4a}{(a+1)\rho_0}\}$,
 and for each $k\in\mathbb{N}$ define
 \begin{equation}\label{FLambdak}
  F^{k}:=\Big\{i\!: \big||\beta_i^k|-|\beta_i^*|\big|\ge\frac{1}{\rho_k}\Big\}
  \ {\rm and}\  \Lambda^k:=\Big\{i\!:|\beta_i^*|\le \frac{4a}{(a\!+\!1)\rho_k}\Big\}.
 \end{equation}
 From Lemma \ref{FLambdk} in Appendix B, the value $\max_{i\in S^*}(1-w_i^k)$
 is upper bounded by $$\max_{i\in S^*}\max(\mathbb{I}_{\Lambda^k}(i),\mathbb{I}_{F^k}(i)).$$
 By this, we have the following conclusion.
 \begin{theorem}\label{error-bound2}
  Suppose that Assumption \ref{assump1} holds, that $X$ has the $\kappa$-RSC over
  $\mathcal{C}(S^*)$, and that the noise $\varepsilon$ is nonzero.
  If $\lambda$ is chosen as in Theorem \ref{error-bound1} and
  the parameter $\rho_3$ satisfies
  $\rho_3 \leq \frac{1}{c\overline{\tau}\lambda\|\varepsilon\|_\infty(\sqrt{4.5s^*}+\!\sqrt{3}/8)}$,
  then for each $k\in\mathbb{N}$
  \begin{align}\label{betak}
   \!\|\beta^{k}\!-\beta^*\|
   &\le \frac{(3+\!\sqrt{3})c\overline{\tau}^2\sqrt{s^*}\|X\|_1\!\|\varepsilon\|_{\infty}}{n}
        +\frac{(3+\!3\sqrt{3})c\overline{\tau}\lambda\sqrt{s^*}\|\varepsilon\|_{\infty}}{2\sqrt{2}}
          \max_{i\in S^*}\mathbb{I}_{\Lambda^{0}}(i)\nonumber\\
  &\quad +c\overline{\tau}\|\varepsilon\|_{\infty}\sqrt{s^*}
           \sum_{j=0}^{k-2}r_{k-j}\Big(\frac{1}{\sqrt{3}}\Big)^{j}
          +\Big(\frac{1}{\sqrt{3}}\Big)^{k-1}\big\|\beta^{1}\!-\beta^*\big\|
 \end{align}
 where we stipulate that $\sum_{j=0}^{k-2}r_{k-j}(\frac{1}{\sqrt{3}})^{j}=0$ for $k=1$.
 \end{theorem}
 \begin{remark}
  {\bf(i)} The error bound in \eqref{betak} consists of
  the statistical error due to the noise,
  the identification error $\max_{i\in S^*}\mathbb{I}_{\Lambda^{0}}(i)$
  related to the choice of $a$ and $\rho_{0}$, and the computation errors
  $\sum_{j=0}^{k-2}r_{k-j}(\frac{1}{\sqrt{3}})^{j}$ and
  $(\frac{1}{\sqrt{3}})^{k-1}\|\beta^{1}\!-\beta^{*}\|$.
  By the definition of $\Lambda^{0}$, when $\rho_{0}$ and $a$ are
  such that $\frac{(a+1)\rho_{0}}{4a}>\frac{1}{\min_{i\in S^*}\!|\beta_i^*|}$,
  the identification error becomes zero. If $\min_{i\in S^*}\!|\beta_i^*|$
  is not too small, it would be easy to choose such $\rho_{0}$.
  Clearly, when $\rho_0$ and $a$ are chosen to be larger, the identification
  error is smaller. However, when $\rho_0$ and $a$ are larger, $\rho_1$ becomes larger
  and each component of $w^1$ is close to $1$ by \eqref{wik}.
  Consequently, it will become very conservative to cut those smaller entries
  of $\beta^2$ when solving the second subproblem. Hence, there is a trade-off
  between the choice of $a$ and $\rho_0$ and
  the computation speed of Algorithm \ref{Alg}.

  \noindent
  {\bf(ii)}
  If the subproblem \eqref{subprob-betak} could be solved exactly,
  the computation error $\sum_{j=0}^{k-2}r_{k-j}(\frac{1}{\sqrt{3}})^{j}$ vanishes.
  If the subproblem \eqref{subprob-betak} is solved with the accuracy $r_k$
  satisfying $r_k\le (\frac{1}{\sqrt{3}})^k\frac{1}{k^{\nu}}$ for
  $\nu>1$, this computation error will tend to $0$ as $k\to+\infty$.
  Since the third term on the right hand side of \eqref{betak} is
  the combination of the noise and $\sum_{j=0}^{k-2}r_{k-j}(\frac{1}{\sqrt{3}})^{j}$,
  it is strongly suggested that the subproblem \eqref{subprob-betak}
  is solved as well as possible.

  For the RSC assumption in Theorem \ref{error-bound1}-\ref{error-bound2}, from \cite{Raskutti10} we know that
  if $X$ is from the $\Sigma_x$-Gaussian ensemble (i.e., $X$ is formed by
  independently sampling each row $x_i^{\mathbb{T}}\sim N(0,\Sigma_x)$, there exists
  a constant $\kappa>0$ (depending on $\Sigma_x$) such that the RSC holds
  on $\mathcal{C}(S^*)$ with probability greater than $1\!-c_1\exp(-c_2n)$
  as long as $n>c_0s^*\log p$, where $c_0,c_1$ and $c_2$ are absolutely positive constants.
  From \cite{Banerjee15}, for some sub-Gaussian $X$,
  the RSC holds on $\mathcal{C}(S^*)$ with a high probability when
  $n$ is over a threshold  depending on the Gaussian
  width of $\mathcal{C}(S^*)$.
 \end{remark}

 \section{Proximal dual semismooth Newton method}\label{sec5}

  By Remark \ref{remark-Alg} (iv), the pivotal part of Algorithm \ref{Alg} is
 the exact solution of
  \begin{equation}\label{weight-l1}
   \min_{\beta\in\mathbb{R}^p}\big\{f_{\tau}(y-\!X\beta)+h_{k-1}(\beta)-\langle\delta^k,\beta-\beta^{k-1}\rangle\big\}
  \end{equation}
  where, for each $k\in\mathbb{N}$, $h_k$ is the function defined in \eqref{hkfun}.
  In this section, we develop a proximal dual semismooth
  Newton method (PDSN) for \eqref{weight-l1}, which is a proximal point algorithm
  (PPA) with the subproblems solved by applying the semismooth Newton method to their dual problems.
 \begin{algorithm}[h]
 \caption{\label{PPA}{\bf\ PPA for solving problem \eqref{weight-l1}}}
 \textbf{Initialization:} Fix $k$. Choose $\gamma_{1,0},\gamma_{2,0},\underline{\gamma}>0,
 \varrho\!\in(0,1)$. Let $\beta^0=\beta^{k-1}$.\\
 \textbf{for} $j=0,1,2,\ldots$.
 \begin{itemize}
   \item[1.] Seek the unique minimizer $\beta^{j+1}$ to the following convex program
                \[
                 \!\min_{\beta\in\mathbb{R}^p}\Big\{f_{\tau}(y-\!X\beta)+h_{k-1}(\beta)
                 -\langle\delta^k,\beta-\!\beta^{k-1}\rangle
                   +\frac{\gamma_{1,j}}{2}\|\beta-\beta^j\|^2+\frac{\gamma_{2,j}}{2}\|X(\beta-\!\beta^j)\|^2\Big\}.
                \]

   \item[2.] If $\beta^{j+1}$ satisfies the stopping rule, then stop. Otherwise,
                update $\gamma_{1,j}$ and $\gamma_{2,j}$ by
                $\gamma_{1,j+1}=\max(\underline{\gamma},\varrho \gamma_{1,j})$ and
                $\gamma_{2,j+1}=\max(\underline{\gamma},\varrho\gamma_{2,j})$.
 \end{itemize}
 \vspace{-0.3cm}
 \textbf{end for}
 \end{algorithm}
 \begin{remark}\label{remark-PPA}
  {\bf(i)} Since $f_{\tau}(y\!-\!X\cdot)$ and $h_{k-1}$ are convex
  but nondifferentiable, we follow the same line as in \cite{Tang19} to
  introduce a key proximal term $\frac{\gamma_{2,j}}{2}\|X\beta-\!X\beta^j\|^2$
  except the common $\frac{\gamma_{1,j}}{2}\|\beta-\beta^j\|^2$. As will
  be shown later, this provides an effective way to handle the nonsmooth $f_{\tau}(y-\!X\cdot)$.

  \noindent
  {\bf(ii)} The first-order optimality conditions for \eqref{weight-l1}
  have the following form
  $
    u\in\partial\!f_{\tau}(z),\,
    X^{\mathbb{T}}u+\delta^k\in\partial h_{k-1}(\beta),\,
    y-\!X\beta-\!z=0,
  $
  where $u\in\mathbb{R}^n$ is the multiplier vector associated to
  $y-X\beta-z=0$. Hence, the KKT residual of problem \eqref{weight-l1}
  at $(\beta^{j},z^{j},u^{j})$ can be measured by
  \[
    {\bf Err}_{\rm PPA}^{j}\!:=\!
    \frac{\sqrt{\|z^{j}\!-\!\mathcal{P}\!f_{\tau}(z^{j}\!+\!u^{j})\|^2
    +\!\|\beta^{j}\!-\!\mathcal{P}h_{k-1}(X^{\mathbb{T}}u^{j}\!+\!\delta^{k})\|^2
    +\!\|y-\!X\beta^{j}\!-\!z^{j}\|^2}}{1+\|y\|}.
  \]
  So, we suggest ${\bf Err}_{\rm PPA}^{j}\!\le\epsilon_{\rm PPA}^{j}$
  as the stopping condition of Algorithm \ref{PPA}.
  \end{remark}

 The efficiency of Algorithm \ref{PPA} depends on the solution of its subproblem,
  which by introducing a variable $z\in\mathbb{R}^n$ is equivalently written as
 \begin{align}\label{Esubprobj}
  &\min_{\beta\in\mathbb{R}^p,z\in\mathbb{R}^n}\Big\{f_{\tau}(z)+h_{k-1}(\beta)
  -\!\langle\delta^k,\beta-\!\beta^{k-1}\rangle               +\frac{\gamma_{1,j}}{2}\|\beta-\beta^j\|^2+\frac{\gamma_{2,j}}{2}\|z-z^j\|^2\Big\}\nonumber\\
  &\quad\ {\rm s.t.}\quad X\beta+z-y=0\ \ {\rm with}\ \ z^j=y-\!X\beta^j.
 \end{align}
 After an elementary calculation, the dual of \eqref{Esubprobj} takes the following form
 \begin{equation*}\label{subdprobj}
  \min_{u\in\mathbb{R}^{n}}\bigg\{\Psi_{k,j}(u)\!:=\frac{\|u\|^2}{2\gamma_{2,j}} -e_{\gamma_{2,j}^{-1}}f_{\tau}\Big(z^j-\frac{u}{\gamma_{2,j}}\Big)
   -e_{\gamma_{1,j}^{-1}}h_{k-1}\Big(\beta^j-\frac{X^{\mathbb{T}}u\!+\!\delta^k}{\gamma_{1,j}}\Big)
   +\frac{\|X^{\mathbb{T}}u\|^2}{2\gamma_{1,j}}\bigg\}.
  \end{equation*}
  Since $\Psi_{k,j}$ is a smooth convex function,
  seeking an optimal solution of the last dual problem
  is equivalent to finding a root to the system
  \begin{equation}\label{semismooth-system}
   \Phi_{k,j}(u):=-\mathcal{P}_{\gamma_{2,j}^{-1}}f_{\tau}\Big(z^j\!-\!\frac{u}{\gamma_{2,j}}\Big)
   -X\mathcal{P}_{\gamma_{1,j}^{-1}}h_{k-1}\Big(\beta^j\!-\!\frac{X^{\mathbb{T}}u\!+\!\delta^k}{\gamma_{1,j}}\Big)+y=0.
  \end{equation}
  Since $\mathcal{P}_{\gamma_{2,j}^{-1}}f_{\tau}$ and $\mathcal{P}_{\gamma_{1,j}^{-1}}h_{k-1}$
  are strongly semismooth by Appendix A and the composition of
  strongly semismooth mappings is strongly semismooth by \cite{Facchinei03},
  the mapping $\Phi_{k,j}$ is strongly semismooth. Inspired by this,
  we use the semismooth Newton method to seek a root to system
  \eqref{semismooth-system}, which by \cite{QiSun93} is expected to have
  a superlinear even quadratic convergence rate.
  By Proposition 2.3.3 and Theorem 2.6.6 of \cite{Clarke83},
  the Clarke Jacobian $\partial_C\Phi_{k,j}(u)$ of $\Phi_{k,j}$ at $u$ is included in
  \begin{align}\label{inclusion}
   &\gamma_{2,j}^{-1}\partial_C\big[\mathcal{P}_{\gamma_{2,j}^{-1}}f_{\tau}\big]\Big(z^j\!-\!\frac{u}{\gamma_{2,j}}\Big)
   +\!\gamma_{1,j}^{-1}X\partial_C\big[\mathcal{P}_{\gamma_{1,j}^{-1}}h_{k-1}\big]
   \Big(\beta^j\!-\!\frac{X^{\mathbb{T}}u\!+\delta^k}{\gamma_{1,j}}\Big)X^{\mathbb{T}}\nonumber\\
   &=\gamma_{2,j}^{-1}\mathcal{U}_j(u)+\gamma_{1,j}^{-1}X\mathcal{V}_j(u)X^{\mathbb{T}}
   \ \forall u\in\mathbb{R}^n
  \end{align}
  where \eqref{inclusion} is due to Lemma \ref{CJacobi-halpha}-\ref{CJacobi-ftau} in Appendix A,
  and $\mathcal{U}_j(u)$ and $\mathcal{V}_j(u)$ are
  \begin{align*}
   \mathcal{U}_j(u):=\Big\{{\rm Diag}(v_1,\ldots,v_n)\ |\ v_i\in\partial_C\big[\mathcal{P}_{\gamma_{2,j}^{-1}}(n^{-1}\theta_{\tau})\big](z_i^j-\gamma_{2,j}^{-1}u_i)\Big\},
   \qquad\\
   \mathcal{V}_j(u)\!:=\!\Big\{{\rm Diag}(v)\,|\,
   v_i=1\ {\rm if}\ |(\gamma_{1,j}\beta^j\!-\!X^{\mathbb{T}}u-\!\delta^k)_i|>\omega_i^k,
   {\rm  otherwise}\ v_i\in[0,1]\Big\}.
  \end{align*}
  For each $U^j\!\in\mathcal{U}_j(u)$ and $V^j\!\in\mathcal{V}_j(u)$,
  the matrix $\gamma_{2,j}^{-1}U^j+\!\gamma_{1,j}^{-1}XV^jX^{\mathbb{T}}$
  is semidefinite, and positive definite when
  $\{i\ |\ \frac{\tau-1}{n\gamma}\!\le z_i^j-\gamma_{2,j}^{-1}u_i\le\!\frac{\tau}{n\gamma}\}=\emptyset$
  or the matrix $X_{J}$ has full row rank with $J=\!\{i\ |\ |(\gamma_{1,j}\beta^j-X^{\mathbb{T}}u-\delta^k)_i|>\omega_i^k\}$.
  To ensure that each iterate of the semismooth Newton method works,
  or each element of Clarke Jacobian $\partial_C\Phi_{k,j}(u)$ is nonsingular,
  we add a small positive definite perturbation $\mu I$ to
  $\gamma_{2,j}^{-1}U^j+\!\gamma_{1,j}^{-1}XV^jX^{\mathbb{T}}$.
  The detailed iterates of the semismooth Newton method is provided in Appendix C.
 \section{Numerical experiments}\label{sec6}

We shall test the performance of Algorithm \ref{Alg} with the subproblems
  solved by PDSN, SeDuMi and sPADMM, respectively, on synthetic and real data,
  and call the three solvers MSCRA\_PPA, MSCRA\_IPM and MSCRA\_ADMM, respectively.
  Among others, SeDuMi is solving the equivalent LP of \eqref{subprob-betak}:
  \begin{align}\label{equiv-LP}
   &\min_{(\beta^{+},\beta^{-})\in\mathbb{R}_{+}^{2p},(\zeta^{+},\zeta^{-})\in\mathbb{R}_{+}^{2n}}
    \langle\omega^k,\beta^{+}\rangle+\langle\omega^k,\beta^{-}\rangle
    +\frac{\tau}{n}\langle\zeta^{+},e\rangle+\frac{1-\tau}{n}\langle\zeta^{-},e\rangle\nonumber\\
   &\qquad\qquad {\rm s.t.}\ \ X\beta^{+}-X\beta^{-}+\zeta^{+}-\zeta^{-}=y,
  \end{align}
  and the iterates of sPADMM are described in Appendix C. All numerical
  results are computed by a laptop computer running on 64-bit Windows
  System with an Intel(R) Core(TM) i7-8565 CPU 1.8GHz and 8 GB RAM.

  For SeDuMi, we adopt the default setting, and for sPADMM we choose
  the step-size $\varrho=1.618$ and the initial $\sigma=1$, and adopt
  the stopping criterion in Appendix C with $j_{\rm max}=3000$
  and $\epsilon_{\rm ADMM}=10^{-6}$. For PDSN, we choose
  $\underline{\gamma}=10^{-8},\varrho=5/7$ and $\gamma_{1,0}=\gamma_{2,0}=\min(0.1,R_0)$
  where $R_0$ is the relative KKT residual at the initial $(\beta^0,z^0,u^0)$,
  and adopt the stopping criterion in Remark \ref{remark-PPA}(ii)
  with $\epsilon_{\rm PPA}^{j+1}=\max(10^{-8},0.1\epsilon_{\rm PPA}^{j})$
  for $\epsilon_{\rm PPA}^{0}\!=10^{-6}$ and the stopping rule
  $\frac{\|\Phi_{k,j}(u^l)\|}{1+\|y\|}\le0.1\epsilon_{\rm PPA}^{j}$ for Algorithm 1 in Appendix C.

  For MSCRA\_IPM, MSCRA\_ADMM and MSCRA\_PPA, we use $w^0=0$,
  and terminate them at $\beta^k$ when $k>10$, or
  $N_{\rm nz}(\beta^{k})=\cdots=N_{\rm nz}(\beta^{k-3})$ and ${\bf Err}_k\le 10^{-5}$,
  or $N_{\rm nz}(\beta^{k})=\cdots=N_{\rm nz}(\beta^{k-2})$ and
  $|{\bf Err}_k-{\bf Err}_{k-2}|\le10^{-6}$, where $N_{\rm nz}(\beta^k)\!:=\!\sum_{i=1}^p\mathbb{I}
  \big\{|\beta_i^k|>\!10^{-6}\max(1,\|\beta^k\|_\infty)\big\}$ denotes the number
  of nonzero entries of $\beta^k$, and ${\bf Err}_k$ is the KKT residual
  at the $k$th step defined in \eqref{errk}. We update $\rho_k$ by $\rho_1=\max\big(1,\frac{1}{3\|\beta^1\|_{\infty}}\big)$
  and $\rho_k=\min\big(\frac{5}{4}\rho_{k-1},\frac{10^8}{\|\beta^k\|_\infty}\big)$
  for $k=2,3$. In addition, during the implementation of three solvers, we run SeDuMi,
  sPADMM and PSDN to solve the $k$th subproblem with the optimal
  solution of the $(k\!-\!1)$th subproblem yielded by them as the starting point.
  When $k=1$, we choose $\beta^0=0$ to be the starting point of MSCRA\_IPM and MSCRA\_ADMM,
  and use $\beta^0=0$ to run Algorithm \ref{PPA}.

 \noindent{\bf 6.1. Comparisons of three solvers for the subproblem}\label{sec6.1}

 We make numerical comparisons among SeDuMi, sPADMM and PDSN by
 applying them to the problem \eqref{subprob-betak} for $k=1$, i.e., the $\ell_1$-regularized
 check loss minimization problem. Inspired by the work owing to \cite{Gu18},
 we consider the simulation model $y_i =x_i^{\mathbb{T}}\beta^* +\kappa\varepsilon_i$
 for $i=1,\ldots,n$ in \cite{Friedman10} to generate data,
 where $x_i^{\mathbb{T}}\sim N(0,\Sigma)$ for $i=1,\ldots,n$ with
 $\Sigma=(\alpha+(1-\!\alpha)\mathbb{I}_{\{i=j\}})_{p\times p},
 \beta_j^*=\!(-1)^j\exp(-\frac{2j-1}{20})$, $\varepsilon\sim N(0,\Sigma)$,
 and $\kappa$ is chosen such that the signal-noise ratio of
 the data is $3.0$. We focus on the high-dimensional situation with
 $(p,n)=(5000,500)$ and $\alpha=0$ and $0.95$. Figure \ref{fig1}-\ref{fig2}
 show the optimal values yielded by three solvers and their CPU time
 (in seconds) on solving \eqref{subprob-betak} with $k=1$ and the same sequence
 of $50$ values of $\lambda$. By the results in Section \ref{sec4},
 we select the $50$ values of $\lambda$ by
 \begin{equation}\label{lambda}
  \lambda_i=\max\big(0.01, \gamma_i\|X\|_1/n \big)
  \ \ {\rm with}\ \ \gamma_i=\gamma_{\rm min}+((i-1)/49)(\gamma_{\rm max}-\gamma_{\rm min})
 \end{equation}
 for $i=1,2,\ldots,50$, where $\gamma_{\rm min}=0.02$, and $\gamma_{\rm max}=0.25$
 and $0.38$ respectively for $\alpha=0$ and $0.95$. Such $\gamma_{\rm max}$
 is such that $N_{\rm nz}(\beta^{f})$ attains the value $0$,
 where $\beta^f$ represents the final output of a solver.
\begin{figure}[h]
  \centerline{\epsfig{file=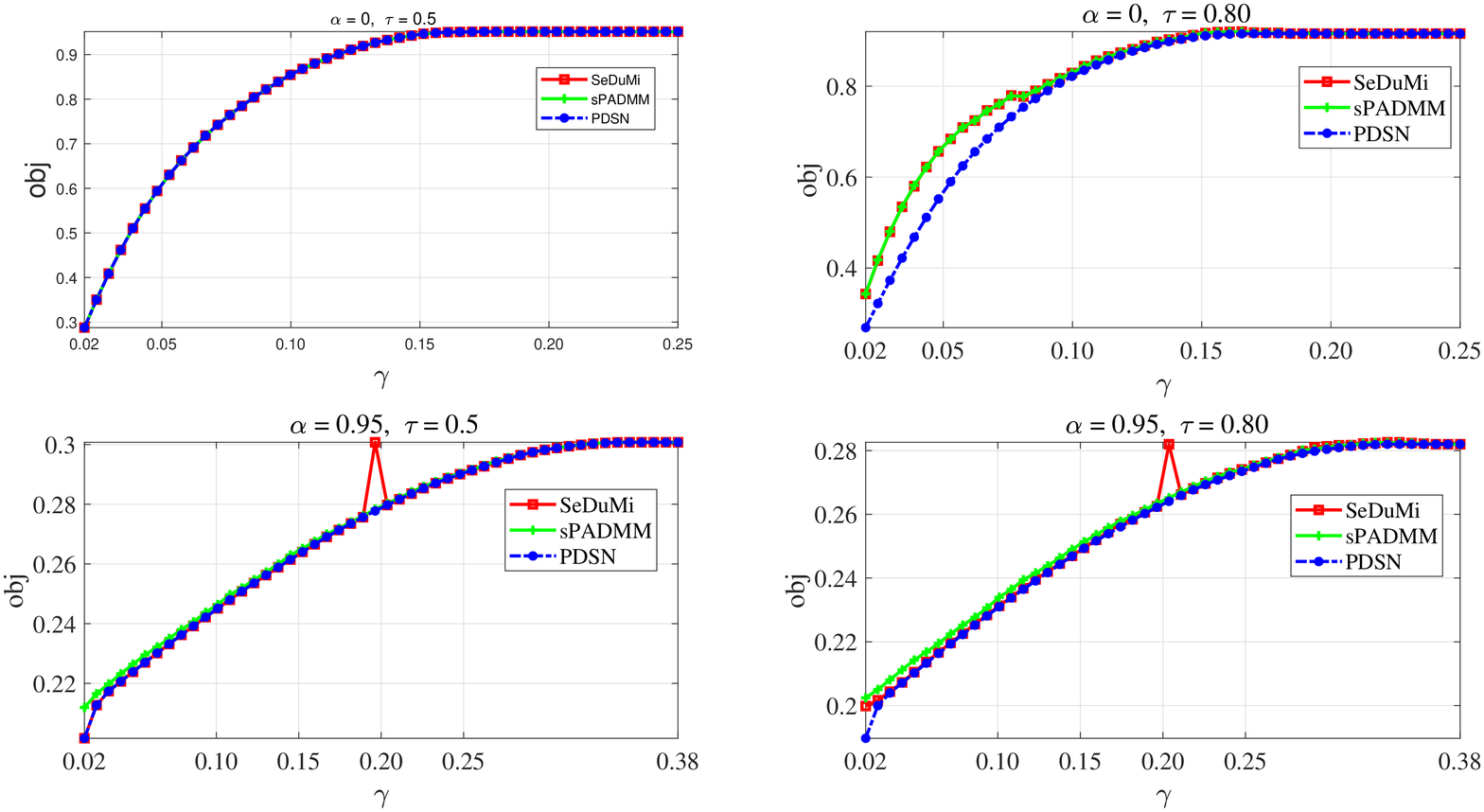,width=6in}}\par
   \vspace{-0.8cm}
  \caption{\small Optimal values of three solvers for the sample size $n=500$}
  \label{fig1}
 \end{figure}
 \begin{figure}[h]
  \centerline{\epsfig{file=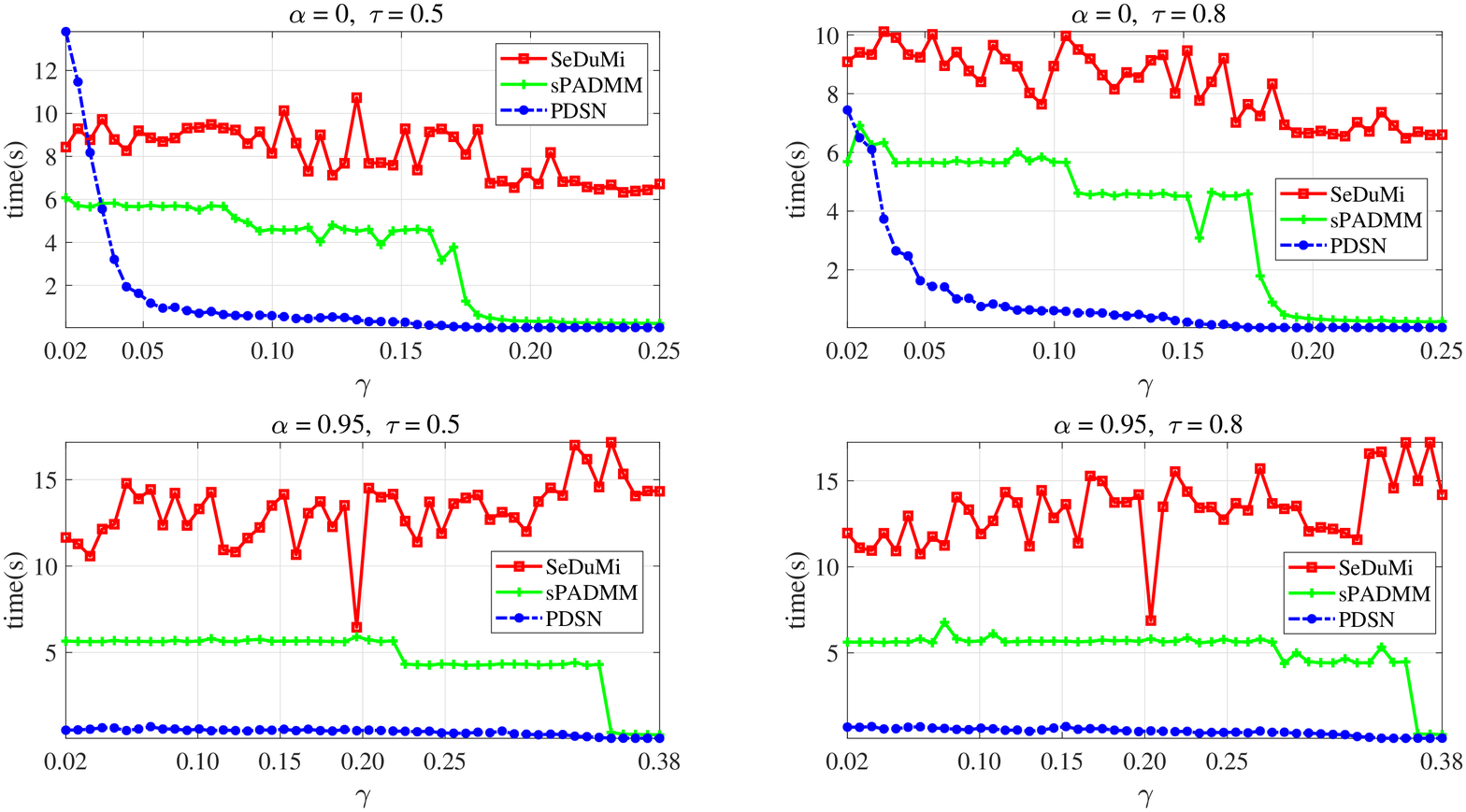,width=6in}}\par
   \vspace{-0.8cm}
  \caption{\small CPU times of three solvers for the sample size $n=500$}
  \label{fig2}
 \end{figure}

 Figure \ref{fig1} shows that the three solvers yield comparable optimal
 values, and the optimal values given by PDSN are a little better than
 those given by SeDuMi and sPADMM. Figure \ref{fig2} shows that PDSN requires
 much less CPU time than SeDuMi and sPADMM do, and for $\alpha=0.95$ the CPU time
 of the former is on average about $0.03$ and $0.09$
 times that of SeDuMi and sPADMM, respectively, but for $\alpha=0,\tau=0.5$,
 when $\lambda<\lambda_3$, PDSN requires more CPU time since the Clarke Jacobians
 are close to singularity. This shows that if the parameter $\lambda$ in the model
 is not too small (a common setting for sparsity), PDSN is superior to SeDuMi
 and sPADMM in terms of the optimal value and CPU time. We find that
 sPADMM always attains the maximum number of iterations $3000$ for all test problems
 (it even attains the maximum number of iterations if $j_{\rm max}=10000$).
 Since $j_{\rm max}=3000$ is used here, its CPU time is less than that of SeDuMi.

 \noindent{\bf 6.2. Numerical performance of Algorithm \ref{Alg}}\label{sec6.2}

 We first apply MSCRA\_PPA to the example in Section 3.1 of \cite{Wang12},
 i.e., solve \eqref{Eprob-penalty} with $\nu=\lambda^{-1}$ for
 $\lambda=\max(0.01,0.1\|X\|_1/n)$, for which the scalar response is
 generated according to the heteroscedastic location-scale model
 $Y=X_6+X_{12}+X_{15}+X_{20}+0.7X_1\varepsilon$, where $\varepsilon\sim N(0,1)$
 is independent of the covariates. Table \ref{identification} reports
 its identification performance for $\tau=0.3,0.5$ and $0.7$ under
 different sample size, where \textbf{Size}, \textbf{AE},
 $P_1$ and $P_2$ have the same meaning as in \cite{Wang12}.
 We see that, for $\tau=0.5$, $P_2$ always equals $0$. So, the check loss with $\tau=0.5$
 can not identify $X_1$, but the check loss with $\tau=0.3$ and $0.7 $
 can identify $X_1$ and the proportion of identifying $X_1$ increases
 as $n$ becomes large.

 \begin{table}[h]
 \caption{  Identification performance of MSCRA\_PPA}
 \label{identification}
 \centering
 \scalebox{0.95}{
 \begin{tabular}{c|ccccc }
  \hline
  \hline
   & &$n=250$&$n=300$&$n=400$ &$n=500$  \\
   \hline
  \multirow{4}{*}{$\tau=0.3$}
  & \rm{Size} & 11.800(4.369) & 9.320(3.146)  &6.290(1.472)  &5.330(0.697) \\
  & $P_1$     & 0.81 & 0.83  & 0.93 & 0.91\\
  & $P_2$     & 0.81 &0.83   &0.93  &0.91 \\
  & \rm{AE}   & 0.197(0.174) &0.170(0.165)    & 0.176(0.155)  &0.145(0.127) \\
\hline
  \multirow{4}{*}{$\tau=0.5$}
  & \rm{Size} & 10.960(3.075) & 7.910(2.060)  &5.270(1.171)  & 4.370(0.597)\\
  & $P_1$     &1.00 & 1.00  & 1.00 &1.00 \\
  & $P_2$     &0.00 & 0.00  & 0.00 & 0.00\\
  & \rm{AE}   &0.034(0.014) &0.027(0.011)   & 0.021(0.010) &0.018(0.008) \\
  \hline
  \multirow{4}{*}{$\tau=0.7$}
  & \rm{Size} &12.590(4.356)  &8.320(2.169)   & 6.310(1.308) &5.380(0.693) \\
  & $P_1$     & 0.79 & 0.88  & 0.91 &0.93 \\
  & $P_2$     & 0.79 & 0.88  & 0.91 &0.93  \\
  & \rm{AE}   & 0.183(0.175)  & 0.220(0.180)   &0.151(0.146)  &0.162(0.142) \\
\hline
\hline
\end{tabular}}
\end{table}

 Next we use a synthetic example to show that MSCRA\_PPA can solve efficiently
 a series of zero-norm regularized problems \eqref{prob} with different $\tau$
 but a fixed $\lambda$. We generate an i.i.d. standard normal random vector
 $\beta_{S^*}^*$ with $s^*=\lfloor 0.5\sqrt{p}\rfloor$ entries of $S^*$ chosen
 randomly from $\{1,\ldots,p\}$ for $p=15000$, and then obtain the response vector
 $y$ from model \eqref{QR-model}, where $x_i^{\mathbb{T}}\sim N(0,\Sigma)$ for
 $i=1,\ldots,n$ with $\Sigma=0.6E+0.4I$ and $n=\lfloor 2s^*\log p\rfloor$,
 and the noise $\varepsilon_i$ is from the Laplace distribution with density
 $d(u)=0.5\exp(-|u|)$. Here, $E$ is a $p\times p$ matrix of all ones.
 Figure \ref{fig3} describes the average absolute $\ell_2$-error $\|\widehat{\beta}^f\!-\!\beta^*\|$
 and time when applying MSCRA\_PPA to $10$ test problems
 for $\tau\in\{0.05,0.1,0.15,\ldots,0.95\}$ with $\nu=\lambda^{-1}$ and $\lambda=37.5/n$.
 We see that MSCRA\_PPA yields better $\ell_2$-errors for $\tau$ close to $0.5$,
 and worse $\ell_2$-errors for $\tau$ close to $0$ or $1$. So,
 for this class of noises, the check loss with $\tau$ close to $0.5$ is suitable.
 The MSCRA\_PPA yields a desired solution for all test problems in $40$ seconds,
 and the CPU time for $\tau$ close to $0$ or $1$ is about $1.5$ times that of
 $\tau$ close to $0.5$. This means that it is an efficient solver for a series of
 zero-norm regularized problems in \eqref{prob}.

 \begin{figure}[H]
  \centerline{\epsfig{file=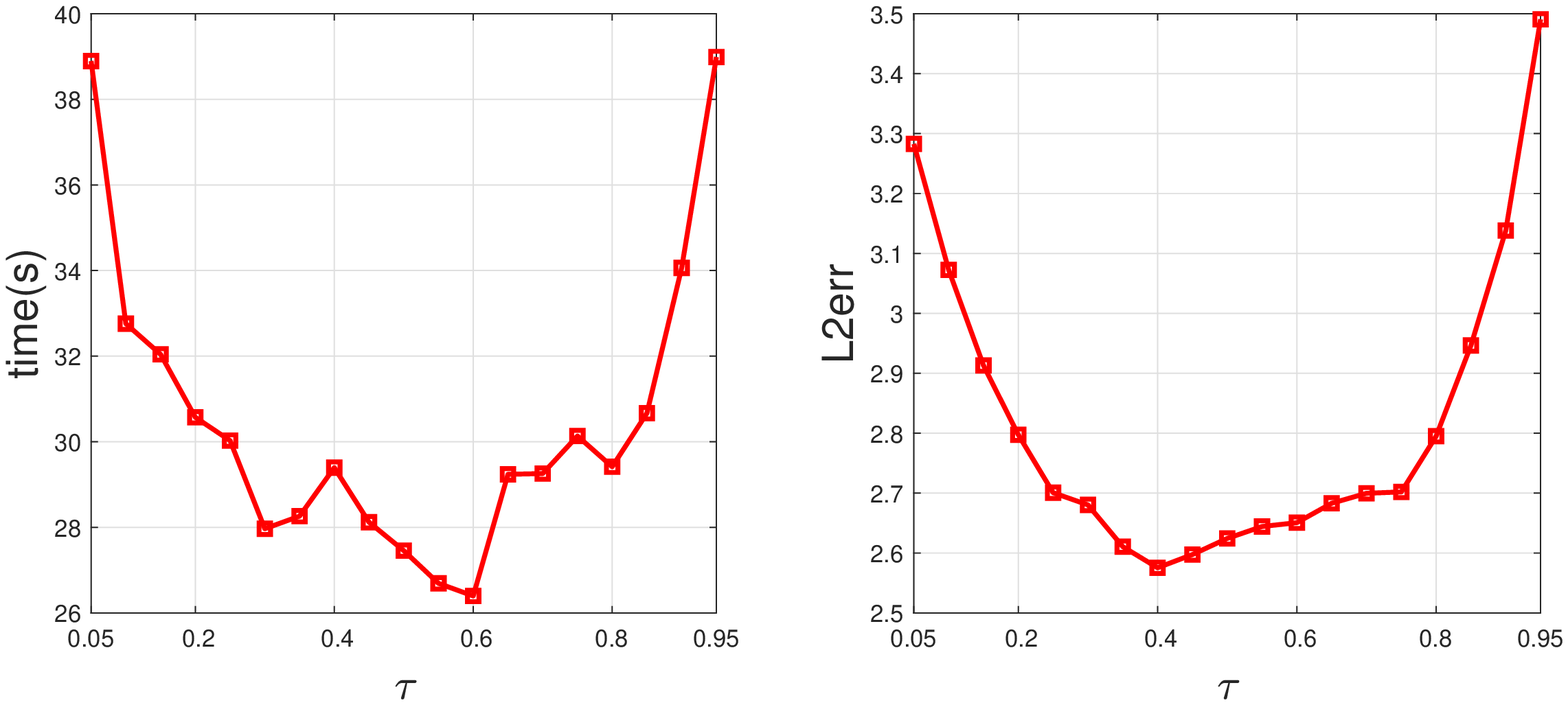,width=6in}}\par
   \vspace{-0.8cm}
  \caption{\small Performance of MSCRA\_PPA under different quantile level $\tau$}
  \label{fig3}
 \end{figure}
 \section{Conclusions}\label{sec7}

  \vspace{-0.3cm}
  We have proposed a multi-stage convex relaxation approach, MSCRA\_PPA,
  for  computing a desirable approximation to the zero-norm penalized QR,
  which is defined as a global minimizer of an NP-hard problem.
  Under the common RSC condition and a mild restriction on the noises,
  we established the error bound of every iterate to the true estimator
  and the linear rate of convergence of the iterate sequence in a statistical sense.
  Numerical comparisons with MSCRA\_IPM and MSCRA\_ADMM show that MSCRA\_PPA
  yields a comparable estimation performance within much less time.
%
%
\vskip 14pt
\noindent {\large\bf Supplementary Materials}

 The online supplementary material consists of five parts. Appendix A includes
 some preliminary knowledge on generalized subdifferentials and Clarke
 Jacobian, and some lemmas used in Section \ref{sec2}-\ref{sec5}; Appendix B includes
 the proof of Theorem \ref{error-bound1} and Theorem \ref{error-bound2};
 Appendix C introduces the semismooth Newton method and the semi-proximal
 ADMM in \cite{Gu16}; Appendix D includes performance comparisons of
 MSCRA\_IPM, MSCRA\_ADMM and MSCRA\_PPA on some synthetic data and real data.

\vskip 14pt
\noindent {\large\bf Acknowledgements}

 The authors would like to give their sincere thanks to two anonymous reviewers
 for their helpful comments. The authors would like to express their sincere thanks
 to Professor Kim-Chuan Toh from National University of Singapore for giving them some help
 on the implementation of Algorithm \ref{PPA} when he visited SCUT.
 This work is supported by the National Natural
 Science Foundation of China under project No. 11971177.

  \newpage
  \noindent\centerline{ {\large\bf Supplementary Materials}}

  \noindent
 {\bf\large Appendix A}

 This part includes some preliminary knowledge on generalized subdifferentials
 and Clarke Jacobian, and some lemmas used in Section 2-5.
 First, we recall from \cite[Definition 8.3]{RW98} the notion of
 the subdifferential of an extended real-valued function.
  \begin{definition}\label{Gsubdiff-def}
   Consider a function $f\!:\mathbb{R}^p\to(-\infty,+\infty]$ and
  $x\in{\rm dom}f$. The regular subdifferential of $f$ at $x$,
  denoted by $\widehat{\partial}\!f(x)$, is defined as
  \[
    \widehat{\partial}\!f(x):=\bigg\{v\in\mathbb{R}^p\ \big|\
    \liminf_{x'\to x\atop x'\ne x}
    \frac{f(x')-f(x)-\langle v,x'-x\rangle}{\|x'-x\|}\ge 0\bigg\};
  \]
  and the (limiting) subdifferential of $f$ at $x$, denoted by $\partial\!f(x)$, is defined as
  \[
    \partial\!f(x)\!:=\!\Big\{v\in\mathbb{R}^p\ |\  \exists\,x^k\!\to x\ {\rm with}\
    f(x^k)\!\to f(x)\ {\rm and}\  v^k\in\widehat{\partial}\!f(x^k)\ {\rm with}\ v^k\!\to v\Big\}.
  \]
 \end{definition}
 \begin{remark}\label{remark-Gsubdiff}
  At each $x\in{\rm dom}f$, $\widehat{\partial}\!f(x)$ and $\partial\!f(x)$
  are closed and satisfy $\widehat{\partial}\!f(x)\subseteq\partial\!f(x)$,
  and the set $\widehat{\partial}\!f(x)$ is convex but $\partial\!f(x)$ is
  generally nonconvex. When $f$ is convex, $\widehat{\partial}\!f(x)=\partial\!f(x)$
  and is precisely the subdifferential of $f$ at $x$ in the sense of
  convex analysis \cite{Roc70}.
 \end{remark}
 \begin{definition}(see \cite{Clarke83})
  Let $H\!:\Omega\to\mathbb{R}^n$ be a locally Lipschitz continuous mapping
  defined on an open set $\Omega\subseteq\mathbb{R}^p$. Denote by
  $D_H\subseteq\Omega$ the set of points where $H$ is differentiable
  and by $H'(z)\in\mathbb{R}^{n\times p}$ the Jacobian of $H$ at $z\in D_H$.
  The Clarke Jacobian of $H$ at $\overline{z}\in\Omega$ is
  \[
    \partial_CH(\overline{z}):={\rm conv}\Big\{\lim_{k\to\infty}H'(z^k)\ |\
    \{z^k\}\subseteq D_H\ {\rm with}\ \lim_{k\to\infty}z^k=\overline{z}\Big\}.
  \]
 \end{definition}

  Generally, it is not easy to characterize the Clarke Jacobian of
  a locally Lipschitz mapping. The following lemmas provide such a
  characterization for the proximal mappings of the weighted $\ell_1$-norm
  and the check loss function.
  \begin{lemma}\label{CJacobi-halpha}
   For a given $\omega\in\mathbb{R}_{+}^p$, let $h(x):=\|\omega\circ x\|_1$ for $x\in\mathbb{R}^p$.
   Then,
   \begin{align*}
     \mathcal{P}_{\gamma^{-1}}h(z)
     ={\rm sign}(z)\max\big(|z|-\gamma^{-1}w,0\big)\quad\ \forall z\in\mathbb{R}^p,\qquad\qquad\\
     \partial_C(\mathcal{P}_{\gamma^{-1}}h)(z)
      =\big\{{\rm Diag}(v_1,\ldots,v_n)\,|\, v_i=1\ {\rm if}\ |\gamma z_i|>\omega_i,\,
      {\rm otherwise}\,v_i\in[0,1]\big\}.
  \end{align*}
 \end{lemma}
  \begin{lemma}\label{CJacobi-ftau}
   For any given $\tau\in(0,1)$, let $\theta_{\tau}$ and $f_{\tau}$ be
   the function defined as in (2.2). Then,
   for any given $\gamma>0$ and $z\in\mathbb{R}^p$, it holds that
   \[
     \big[\mathcal{P}_{\gamma^{-1}}f_{\tau}(z)\big]_i
     =\max\Big(\max\Big(z_i-\frac{\tau}{n\gamma},0\Big),\frac{\tau-1}{n\gamma}-z_i\Big)
     \ \ {\rm for}\ i=1,2,\ldots,p
   \]
  and $\partial_C(\mathcal{P}_{\gamma^{-1}}f_{\tau})(y)=
  \big\{{\rm Diag}(v_1,\ldots,v_n)\ |\ v_i\in\partial_C\big[\mathcal{P}_{\gamma^{-1}}(n^{-1}\theta_{\tau})\big](z_i)\big\}$ with
  \begin{equation}\label{subdiff-thetatau}
   \partial_C\big[\mathcal{P}_{\gamma^{-1}}(n^{-1}\theta_{\tau})\big](t)
    =\left\{\begin{array}{cl}
      \!\{1\} &{\rm if}\ t>\frac{\tau}{n\gamma}\ {\rm or}\ t<\frac{\tau-1}{n\gamma};\\
      \![0,1] &{\rm if}\ t=\frac{\tau}{n\gamma}\ {\rm or}\ \frac{\tau-1}{n\gamma}; \\
      \!\{0\} &{\rm if}\ \frac{\tau-1}{n\gamma}<t<\frac{\tau}{n\gamma}.
   \end{array}\right.
  \end{equation}
  \end{lemma}

 To close this part, we show that under a mild condition, the zero-norm regularized
 composite problem has a nonempty global optimal solution set.
 \begin{lemma}\label{sol-exist}
  Let $A\in\mathbb{R}^{n\times p}$ and $b\in\mathbb{R}^n$ be given,
  and let $g\!:\mathbb{R}^n\to\mathbb{R}$ be an lsc coercive
  function with $\inf_{z\in\mathbb{R}^n}g(z)>-\infty$.
  Then, for any given $\nu>0$, the zero-norm composite problem
  \begin{equation}\label{prob-appendix}
   \min_{x\in\mathbb{R}^p}\Big\{\nu g(b-\!Ax)+\|x\|_0\Big\}
  \end{equation}
  has a nonempty global optimal solution set.
 \end{lemma}
 \begin{proof}
  Notice that the objective function of \eqref{prob-appendix}
  is lower bounded. So, it has an infimum, say $\alpha^*$.
  Then there exists a sequence $\{x^k\}\subset\mathbb{R}^p$ such that
  \begin{equation}\label{gxk-ineq}
    \nu g(b-\!Ax^k)+\|x^k\|_0\le \alpha^*+1/k\ \ {\rm for\ each}\ k.
  \end{equation}
  If $\{x^k\}$ is bounded, then by letting $\overline{x}$ be
  an arbitrary limit point of $\{x^k\}$ and using the lsc of
  $x\mapsto g(b-\!A\,x)$ and $\|\cdot\|_0$, we have
  $\nu g(b-\!A\overline{x})+\|\overline{x}\|_0\le\alpha^*$.
  This shows that $\overline{x}$ is a global optimal solution of
  the problem \eqref{prob-appendix}. Next we consider the case that
  $\{x^k\}$ is unbounded. Define
  \[
    J:=\big\{i\in\{1,\ldots,p\}\ |\ \{x_i^k\}\ {\rm is\ unbounded}\big\}
    \ \ {\rm and}\ \
    \overline{J}:=\{1,\ldots,p\}\backslash J.
  \]
  Along with \eqref{gxk-ineq}, it immediately follows that
  for all sufficiently large $k$,
  \begin{equation}\label{gxk-ineq1}
    \nu g(b-\!Ax^k)+|J|+\|x_{\overline{J}}^k\|_0\le\alpha^*+1/k.
  \end{equation}
  This, by the coerciveness of $g$, means that there is a bounded
  sequence $\{z^k\}\subset\mathbb{R}^n$ such that $z^k=b-\!Ax^k$.
  Clearly, $A_Jx_J^k=b-z^k-A_{\overline{J}}x_{\overline{J}}^k$.
  Notice that $\{z^k\}$ and $\{x_{\overline{J}}^k\}$ are bounded.
  We may assume (taking a subsequence if necessary) that $\{z^k\}$
  and $\{x_{\overline{J}}^k\}$ are convergent, say, $z^k\to z^*$
  and $x_{\overline{J}}^k\to\xi^*$.
  Notice that for each $k$, $x_{\!J}^k$ is a solution of the system
  $A_Jy=b-z^k-A_{\overline{J}}x_{\overline{J}}^k$, that is,
  $\{b-z^k-A_{\overline{J}}x_{\overline{J}}^k\}\subset A_J(\mathbb{R}^{|J|})$.
  Together with the closedness of the set $A_J(\mathbb{R}^{|J|})$,
  it follows that $b-z^*-A_{\overline{J}}\xi^*\in A_J(\mathbb{R}^{|J|})$.
  So, there exists $u^*\in\mathbb{R}^{|J|}$ such that
  $A_Ju^*=b-z^*\!-\!A_{\overline{J}}\xi^*$, i.e., $A_Ju^*+A_{\overline{J}}\xi^*\!-\!z^*=b$.
  Taking the limit to the both sides of \eqref{gxk-ineq1}
  and using $b-\!Ax^k=z^k$ gives
  \[
    \nu g(z^*)+|J|+\|\xi^*\|_0\le\alpha^*.
  \]
  Together with $\nu g(b-A_Ju^*-\!A_{\overline{J}}\xi^*)+\|u^*\|_0+\|\xi^*\|_0
  \le \nu g(z^*)+|J|+\|\xi^*\|_0$, we conclude that $(u^*;\xi^*)$ is a global optimal
  solution of the zero-norm composite problem \eqref{prob-appendix}.
 \end{proof}

 \noindent
 {\bf\large Appendix B}

  In this part, for each $k\in\mathbb{N}$ we write $v^k:=e-w^k$
  and $z^{k}\!:=y-\!X\beta^k$. To present the proof of Theorem 2,
  we need the following technical lemma.
 \begin{lemma}\label{betak-lemma1}
  Suppose that Assumption 1 holds and
  for some $k\!\ge 1$ there exists $S^{k-1}\!\supseteq S^*$ with
  \(
    \max_{i\in(S^{k-1})^c}w_i^{k-1}\le\frac{1}{2}.
  \)
  Then, when $\lambda\ge 16\overline{\tau}n^{-1}\|X\|_1+8r_k$,
  \[
    \big\|\Delta\beta^{k}_{(S^{k-1})^c}\|_1\le 3\|\Delta\beta^{k}_{S^{k-1}}\|_1.
  \]
 \end{lemma}
 \begin{proof}
  By the approximate optimality of $\beta^k$ to (3.1)
  and Remark 1(iv),
  \begin{align*}
   f_{\tau}(y-\!X\beta^*)+\lambda\langle v^{k-1},|\beta^*|\rangle
   &\ge f_{\tau}(y-\!X\beta^k)+\lambda\langle v^{k-1},|\beta^k|\rangle+\langle\delta^k,\beta^*-\beta^k\rangle
  \end{align*}
  which, after a suitable rearrangement, takes the following form
  \begin{equation}\label{ineq1-ftau}
   f_{\tau}(y-\!X\beta^{k})-f_{\tau}(y-\!X\beta^*)+\langle\delta^k,\beta^*-\beta^k\rangle
   \le \lambda\langle v^{k-1},|\beta^*|-|\beta^{k}|\rangle.
  \end{equation}
  Recall that $\varepsilon=y-X\beta^*$ and $\|\varepsilon\|_\infty>0$.
  We define the following index sets
  \begin{equation}\label{zk-index}
   \mathcal{I}:=\big\{i\in\{1,\ldots,n\}\!: \varepsilon_i\neq0\big\}\ \ {\rm and}\ \
   \mathcal{J}_k:=\big\{i\notin\mathcal{I}\!: z_i^{k}\neq0\big\}.
  \end{equation}
  By the expression of $f_{\tau}$ and $\theta_{\tau}(0)=0$,
  with the index sets $\mathcal{I}$ and $\mathcal{J}_k$,
  \begin{align}\label{ineq2-ftau}
   &f_{\tau}(y-X\beta^{k})-f_{\tau}(y-X\beta^*)
     =\frac{1}{n}\sum_{i=1}^n[\theta_{\tau}(z^{k}_i)-\theta_{\tau}(\varepsilon_i)]\nonumber\\
   &=\frac{1}{n}\bigg[\sum_{i\in \mathcal{J}_k}\frac{\theta_{\tau}^2(z^{k}_i)-\theta_{\tau}^2(\varepsilon_i)}
    {\theta_{\tau}(z^{k}_i)+\theta_{\tau}(\varepsilon_i)}
   +\sum_{i\in \mathcal{I}}\frac{\theta_{\tau}^2(z^{k}_i)-\theta_{\tau}^2(\varepsilon_i)}
    {\theta_{\tau}(z^{k}_i)+\theta_{\tau}(\varepsilon_i)}\bigg]\nonumber\\
   &\ge\frac{1}{n}\bigg[\sum_{i\in \mathcal{J}_k}\frac{\theta_{\tau}^2(z^{k}_i)-\theta_{\tau}^2(\varepsilon_i)}
    {\overline{\tau}\|z^{k}\|_{\infty}}
   +\sum_{i\in \mathcal{I}}\frac{\theta_{\tau}^2(z^{k}_i)-\theta_{\tau}^2(\varepsilon_i)}
    {\theta_{\tau}(z^{k}_i)+\theta_{\tau}(\varepsilon_i)}\bigg].
  \end{align}
  Notice that $\theta_{\tau}^2$ is smooth and strongly convex of modulus $2\underline{\tau}^2$.
  For each $i$,
  \begin{equation}\label{ineq-theta-tau}
    \theta_{\tau}^2(z^{k}_i)-\theta_{\tau}^2(\varepsilon_i)
    \ge 2(\tau-\mathbb{I}_{\mathbb{R}_{-}}(\varepsilon_i))^2\varepsilon_i(z_i^k-\varepsilon_i)
    +\underline{\tau}^2(z^{k}_i-\varepsilon_i)^2.
  \end{equation}
  This implies that $\theta_{\tau}^2(z^{k}_i)-\theta_{\tau}^2(\varepsilon_i)\ge\underline{\tau}^2(z^{k}_i-\varepsilon_i)^2$
  for each $i\in\mathcal{J}_k$, and then
  \begin{equation}\label{ineq3-ftau}
   \sum_{i\in \mathcal{J}_k}\frac{\theta_{\tau}^2(z^{k}_i)-\theta_{\tau}^2(\varepsilon_i)}
    {\overline{\tau}\|z^{k}\|_{\infty}}
   \ge \frac{\underline{\tau}^2}{\overline{\tau}}
    \sum_{i\in \mathcal{J}_k}\frac{(z^{k}_i-\varepsilon_i)^2}{\|z^{k}\|_{\infty}}.
  \end{equation}
  For each $i\in\mathcal{I}$, write $\widetilde{z}_i^{k}:=\frac{2(\tau-\mathbb{I}_{\mathbb{R}_{-}}(\varepsilon_i))^2 \varepsilon_i}{\theta_{\tau}(z_i^{k})+\theta_{\tau}(\varepsilon_i)}$.
  From \eqref{ineq-theta-tau}, it follows that
 \begin{align}\label{ineq4-ftau}
  \!\sum_{i\in \mathcal{I}}\frac{\theta_{\tau}^2(z^{k}_i)-\theta_{\tau}^2(\varepsilon_i)}
    {\theta_{\tau}(z^{k}_i)+\theta_{\tau}(\varepsilon_i)}
   &\ge\sum_{i\in \mathcal{I}}\widetilde{z}_i^k(z_i^k-\varepsilon_i)+\underline{\tau}^2
     \sum_{i\in \mathcal{I}}\frac{(z^{k}_i-\varepsilon_i)^2}{\theta_{\tau}(z^{k}_i)+\theta_{\tau}(\varepsilon_i)}\nonumber\\
   &\ge-\|\widetilde{z}^k\|_\infty\|X(\beta^k\!-\beta^*)\|_1+\underline{\tau}^2
     \sum_{i\in \mathcal{I}}\frac{(z^{k}_i-\varepsilon_i)^2}{\overline{\tau}(\|z^{k}\|_{\infty}+\|\varepsilon\|_\infty)}\nonumber\\
   &\!\ge -2\overline{\tau}\big\|X(\beta^k\!-\!\beta^*)\big\|_{1}
     +\frac{\underline{\tau}^2}{\overline{\tau}}
     \sum_{i\in \mathcal{I}}\frac{(z^{k}_i-\varepsilon_i)^2}{\|z^{k}\|_{\infty}\!+\|\varepsilon\|_\infty}
 \end{align}
 where the second inequality is by $\theta_{\tau}(z^{k}_i)\le\overline{\tau}\|z^k\|_{\infty}$
 for $i\in\mathcal{I}$, and the last one is since
 $|\widetilde{z}_i^{k}|\le\frac{2(\tau-\mathbb{I}_{\mathbb{R}_{-}}(\varepsilon_i))^2 |\varepsilon_i|}{\theta_{\tau}(\varepsilon_i)}\le 2\overline{\tau}$ for each $i\in\mathcal{I}$.
 Substituting the inequalities \eqref{ineq3-ftau}-\eqref{ineq4-ftau} into \eqref{ineq2-ftau},
 we obtain that
 \begin{align*}
  f_{\tau}(y\!-\!X\beta^{k})-\!f_{\tau}(y-\!X\beta^*)
  &\ge\frac{\underline{\tau}^2}{n\overline{\tau}}
     \sum_{i\in \mathcal{J}_k\cup \mathcal{I}}\frac{(z^{k}_i-\varepsilon_i)^2}
     {\|z^{k}\|_{\infty}+\|\varepsilon\|_\infty}-\!\frac{2\overline{\tau}}{n}\|X(\beta^k\!-\beta^*)\|_1\nonumber\\
  &=\frac{\underline{\tau}^2\|X(\beta^k\!-\beta^*)\|^2}{n\overline{\tau}
     (\|z^{k}\|_{\infty}\!+\!\|\varepsilon\|_\infty)}-\!\frac{2\overline{\tau}}{n}\|X(\beta^k\!-\!\beta^*)\|_1.
  \end{align*}
  Combining this inequality and \eqref{ineq1-ftau} and
  recalling that $\|\delta^k\|\le r_k$, we get
  \begin{align}\label{ftau-mainineq}
  \!\frac{\underline{\tau}^2\|X(\beta^k-\beta^*)\|^2}{n\overline{\tau}(\|z^{k}\|_{\infty}\!+\!\|\varepsilon\|_\infty)}
  &\le \lambda\langle v^{k-1},|\beta^*|-|\beta^{k}|\rangle
       +\frac{2\overline{\tau}}{n}\big\|X(\beta^k\!-\beta^*)\big\|_{1}
       +\langle\delta^k,\beta^k\!-\beta^*\rangle\nonumber\\
  &\le \lambda\Big(\textstyle{\sum_{i\in S^*}}v_i^{k-1}|\Delta\beta_i^k|
       -\textstyle{\sum_{i\in (S^{k-1})^c}}v_i^{k-1}|\Delta\beta_i^k|\Big)\nonumber\\
  &\quad\ +\big(2n^{-1}\overline{\tau}\|X\|_1+r_k\big)\|\beta^k\!-\beta^*\|_1\nonumber\\
  &=\lambda\Big(\textstyle{\sum_{i\in S^*}}v_i^{k-1}|\Delta\beta_i^k|
       -\textstyle{\sum_{i\in (S^{k-1})^c}}v_i^{k-1}|\Delta\beta_i^k|\Big)\\
  &\quad +\big(2n^{-1}\overline{\tau}\!\|X\|_1\!+r_k\big)
   \big(\|\Delta\beta_{S^{k-1}}^k\|_{1}+\!\|\Delta\beta_{(S^{k-1})^{c}}^k\|_{1}\big)\nonumber
 \end{align}
  Since $S^{k-1}\supset S^*$ and $v_i^{k-1}\in[0.5,1]$ for $i\in(S^{k-1})^{c}$,
  from the last inequality,
 \begin{align*}
  \frac{\underline{\tau}^2\|X(\beta^k\!-\beta^*)\|^2}{n\overline{\tau}(\|z^{k}\|_{\infty}\!+\!\|\varepsilon\|_\infty)}
  &\le\textstyle{\sum_{i\in S^{k-1}}}\big(\lambda v_i^{k-1}
  +2n^{-1}\overline{\tau}\|X\|_1+r_k\big)\big|\Delta\beta_i^k\big|\nonumber\\
  &\quad +\textstyle{\sum_{i\in (S^{k-1})^c}}\big(2n^{-1}\overline{\tau}\|X\|_1+r_k-\lambda/2\big)
  \big|\Delta\beta_i^k\big|\nonumber\\
  &\le\big(\lambda +2n^{-1}\overline{\tau}\!\|X\|_1+r_k\big)
   \big\|\Delta\beta_{S^{k-1}}^k\big\|_1\nonumber\\
  &\quad +\big(2n^{-1}\overline{\tau}\|X\|_1+r_k-\lambda/2\big)
  \big\|\Delta\beta_{(S^{k-1})^{c}}^k\big\|_1.\nonumber
 \end{align*}
 By the nonnegativity of the left hand side and the given assumption on $\lambda$,
 \[
   \big\|\Delta\beta_{(S^{k-1})^{c}}^k\big\|_1
   \le\frac{\lambda+2n^{-1}\overline{\tau}\|X\|_1+r_k}
   {0.5\lambda-2n^{-1}\overline{\tau}\|X\|_1-r_k}\big\|\Delta\beta_{S^{k-1}}^k\big\|_1
  \le 3\big\|\Delta\beta_{S^{k-1}}^k\big\|_1.
 \]
 The desired result follows. The proof is then completed.
 \end{proof}
 \begin{lemma}\label{betak-lemma2}
  Suppose that Assumption 1 holds, that $X$ satisfies the $\kappa$-RSC
  over $\mathcal{C}(S^*)$, and that for some $k\ge 1$ there exists an index set
  $S^{k-1}$ with $|S^{k-1}|\le 1.5s^*$ such that $S^{k-1}\supseteq S^*$ and $\max_{i\in(S^{k-1})^c}w_i^{k-1}\le\frac{1}{2}$.
  Then, when $16 \overline{\tau}n^{-1}\|X\|_1+8r_k \le\lambda
  <\frac{\underline{\tau}^2\kappa -2\overline{\tau}\|X\|_{\rm max}
 (2n^{-1}\overline{\tau}\|X\|_1+r_k)|S^{k-1}|}
 {2\overline{\tau}\|X\|_{\rm max}\|v_{S^*}^{k-1}\|_{\infty}|S^{k-1}|}$,
  \[
    \big\|\Delta\beta^{k}\big\|\le\frac{\overline{\tau}\big(\lambda\|v_{S^*}^{k-1}\|_{\infty}
    +2n^{-1}\overline{\tau}\!\|X\|_1+r_k\big)\sqrt{|S^{k-1}|}\|\varepsilon\|_\infty}
     {\underline{\tau}^2\kappa -2\overline{\tau}\|X\|_{\rm max}\big(\lambda\|v_{S^*}^{k-1}\|_{\infty}
      +2n^{-1}\overline{\tau}\!\|X\|_1+r_k\big)|S^{k-1}|}.
  \]
 \end{lemma}
 \begin{proof}
  Notice that $\|z^k\|_\infty+\|\varepsilon\|_\infty
  \le\|X\Delta\beta^k\|_{\infty}+2\|\varepsilon\|_\infty$. So, we have
  \[
   \frac{\underline{\tau}^2\|X(\beta^k-\beta^*)\|^2}{n\overline{\tau}(\|z^{k}\|_{\infty}+\|\varepsilon\|_\infty)}
   \ge\frac{\underline{\tau}^2\|X\Delta\beta^k\|^2}{n\overline{\tau}(\|X\Delta\beta^k\|+2\|\varepsilon\|_\infty)}.
  \]
  Together with \eqref{ftau-mainineq} and
  $v_i^{k-1}\in[0.5,1]$ for $i\in(S^{k-1})^{c}$, it follows that
 \begin{align*}
 \frac{\underline{\tau}^2\|X\Delta\beta^k\|^2}{n\overline{\tau}(\|X\Delta\beta^k\|_{\infty}\!+\!2\|\varepsilon\|_\infty)}
  &\le\lambda\sum_{i\in S^*}v_i^{k-1}|\Delta\beta_i^k|
           -\frac{\lambda}{2}\!\sum_{i\in (S^{k-1})^c}\!|\Delta\beta_i^k|\nonumber\\
  &\quad +\big(2n^{-1}\overline{\tau}\!\|X\|_1+r_k\big)
   \big(\|\Delta\beta_{S^{k-1}}^k\|_{1}+\|\Delta\beta_{(S^{k-1})^{c}}^k\|_{1}\big)\\
  &\le \Big(\lambda\|v_{S^*}^{k-1}\|_{\infty}+2n^{-1}\overline{\tau}\!\|X\|_1+r_k\Big)
     \|\Delta\beta_{S^{k-1}}^k\|_{1}
 \end{align*}
 where the last inequality is due to $\lambda>16n^{-1}\overline{\tau}\|X\|_1+8r_k$.
 By Lemma \ref{betak-lemma1},
 $\|\Delta\beta^{k}_{(S^{k-1})^c}\|_1\le 3\|\Delta\beta^{k}_{S^{k-1}}\|_1$.
 By the given assumption, $\Delta\beta^{k}\in\mathcal{C}(S^*)$.
 From the $\kappa$-RSC property of $X$ on $\mathcal{C}(S^*)$, it follows that
 $\|X\Delta\beta^k\|^2\ge 2n\kappa\|\Delta\beta^k\|^2$. Then, we obtain
 \[
   \frac{2\underline{\tau}^2\kappa\|\Delta\beta^k\|^2}{\overline{\tau}\big(\|X\Delta\beta^k\|_{\infty}\!+2\|\varepsilon\|_\infty\big)}
   \le \Big(\lambda\|v_{S^*}^{k-1}\|_{\infty}+\frac{2\overline{\tau}\!\|X\|_1}{n}+r_k\Big)
   \big\|\Delta\beta_{S^{k-1}}^k\big\|_1.
 \]
  Multiplying this inequality with
  $\overline{\tau}\big(\|X\Delta\beta^k\|_{\infty}\!+2\|\varepsilon\|_\infty\big)$ yields that
  \begin{align*}
   2\underline{\tau}^2\kappa\|\Delta\beta^k\|^2
   &\le \overline{\tau}\big(\|X\Delta\beta^k\|_{\infty}\!+2\|\varepsilon\|_\infty\big)
       \Big(\lambda\|v_{S^*}^{k-1}\|_{\infty}+\frac{2\overline{\tau}\!\|X\|_1}{n}+r_k\Big)
       \big\|\Delta\beta_{S^{k-1}}^k\big\|_1\\
   &\le \overline{\tau}\|X\Delta\beta^k\|_{\infty} \Big(\lambda\|v_{S^*}^{k-1}\|_{\infty}
       +2n^{-1}\overline{\tau}\!\|X\|_1+r_k\Big)
       \big\|\Delta\beta_{S^{k-1}}^k\big\|_1\\
   &\quad + 2\overline{\tau}\|\varepsilon\|_\infty
   \Big(\lambda\|v_{S^*}^{k-1}\|_{\infty}
       +2n^{-1}\overline{\tau}\!\|X\|_1+r_k\Big)
       \big\|\Delta\beta_{S^{k-1}}^k\big\|_1.
  \end{align*}
  Since $\|X\Delta\beta^k\|_{\infty}\le \|X\|_{\rm max}\|\Delta\beta^k\|_1$,
  along with $\|\Delta\beta^{k}_{(S^{k-1})^c}\|_1\le 3\|\Delta\beta^{k}_{S^{k-1}}\|_1$,
  we have $\|X\Delta\beta^k\|_{\infty}\le 4\|X\|_{\rm max}\|\Delta\beta_{S^{k-1}}^k\|_1$.
  Thus, from the last inequality,
  \begin{align*}
   2\underline{\tau}^2\kappa\|\Delta\beta^k\|^2
   &\le 4\overline{\tau}\|X\|_{\rm max}\Big(\lambda\|v_{S^*}^{k-1}\|_{\infty}
    +2n^{-1}\overline{\tau}\!\|X\|_1+r_k\Big)
       \big\|\Delta\beta_{S^{k-1}}^k\big\|_1^2\\
   &\quad +2\overline{\tau}\Big(\lambda\|v_{S^*}^{k-1}\|_{\infty}+2n^{-1}\overline{\tau}\!\|X\|_1+r_k\Big)
    \big\|\Delta\beta_{S^{k-1}}^k\big\|_1\|\varepsilon\|_{\infty}\\
   &\le 4\overline{\tau}\|X\|_{\rm max}\Big(\lambda\|v_{S^*}^{k-1}\|_{\infty}+\frac{2\overline{\tau}\!\|X\|_1}{n}+r_k\Big)
       |S^{k-1}|\big\|\Delta\beta_{S^{k-1}}^k\big\|^2\\
   &\quad +2\overline{\tau}\Big(\lambda\|v_{S^*}^{k-1}\|_{\infty}+2n^{-1}\overline{\tau}\!\|X\|_1+r_k\Big)
    \sqrt{|S^{k-1}|}\big\|\Delta\beta_{S^{k-1}}^k\big\|\|\varepsilon\|_{\infty}\\
   &\le 4|S^{k-1}|\overline{\tau}\|X\|_{\rm max}\Big(\lambda\|v_{S^*}^{k-1}\|_{\infty}
   +\frac{2\overline{\tau}\!\|X\|_1}{n}+r_k\Big)
        \big\|\Delta\beta^k\big\|^2\\
   &\quad +2\overline{\tau}\Big(\lambda\|v_{S^*}^{k-1}\|_{\infty}+\frac{2\overline{\tau}\!\|X\|_1}{n}+r_k\Big)
    \sqrt{|S^{k-1}|}\big\|\Delta\beta_{S^{k-1}}^k\big\|\|\varepsilon\|_{\infty}.
  \end{align*}
  After a suitable rearrangement, this inequality is equivalent to saying that
 \begin{align*}
  &\Big[2\underline{\tau}^2\kappa -4\overline{\tau}\|X\|_{\rm max}\big(\lambda\|v_{S^*}^{k-1}\|_{\infty}
  +2n^{-1}\overline{\tau}\!\|X\|_1+r_k\big)|S^{k-1}|\Big]\|\Delta\beta^k\|^2\\
  &\le 2\overline{\tau}\Big(\lambda\|v_{S^*}^{k-1}\|_{\infty}+2n^{-1}\overline{\tau}\!\|X\|_1+r_k\Big)
    \sqrt{|S^{k-1}|}\big\|\Delta\beta^k\big\|\|\varepsilon\|_{\infty},
 \end{align*}
 which by $\lambda<\frac{\underline{\tau}^2\kappa -2\overline{\tau}\|X\|_{\rm max}
 (2n^{-1}\overline{\tau}\|X\|_1+r_k)|S^{k-1}|}
 {2\overline{\tau}\|X\|_{\rm max}\|v_{S^*}^{k-1}\|_{\infty}|S^{k-1}|}$
 implies the result.
 \end{proof}

 \noindent
 {\bf\large Proof of Theorem 2}
 \begin{proof}
  For each $k\in\mathbb{N}$, let $S^{k-1}\!:=S^*\cup\{i\notin S^*\!:w^{k-1}_i>\frac{1}{2}\}$.
  If $|S^{k-1}|\leq1.5s^*$, by invoking Lemma \ref{betak-lemma2}
  and using the given assumption, we have
  \begin{align}\label{equa-41}
   \!\big\|\beta^{k}\!-\beta^*\big\|
   &\le\frac{\overline{\tau}\big(\lambda\|v_{S^*}^{k-1}\|_{\infty}
      +2n^{-1}\overline{\tau}\|X\|_1+r_k\big)\sqrt{|S^{k-1}|}\|\varepsilon\|_\infty}
     {\underline{\tau}^2\kappa -2\overline{\tau}\|X\|_{\rm max}\big(\lambda\|v_{S^*}^{k-1}\|_{\infty}
      +2n^{-1}\overline{\tau}\|X\|_1+r_k\big)|S^{k-1}| }\nonumber\\
   &\le \frac{\overline{\tau}\big(\lambda\|v_{S^*}^{k-1}\|_{\infty}
      +2n^{-1}\overline{\tau}\|X\|_1+r_k\big)\sqrt{|S^{k-1}|}\|\varepsilon\|_\infty}
     {\underline{\tau}^2\kappa -3\overline{\tau}\|X\|_{\rm max}\big(\lambda
      +2n^{-1}\overline{\tau}\|X\|_1+\epsilon\big) s^* }\nonumber\\
   &\le c\overline{\tau}\big(\lambda\|v_{S^*}^{k-1}\|_{\infty}
      +2n^{-1}\overline{\tau}\|X\|_1+r_k\big)\sqrt{|S^{k-1}|}\|\varepsilon\|_\infty
  \end{align}
   where the second inequality is by the nondecreasing of $t\mapsto \frac{c_2+t}{c_1-t}$
  for constants $c_1,c_2>0$, and the last one is by the restriction on $\lambda$.
  Since $2n^{-1}\overline{\tau}\|X\|_1+r_k\le\!\frac{\lambda}{8}$
  and $\|v_{S^*}^{k-1}\|_{\infty}\le 1$, it follows that
  $\big\|\beta^{k}-\beta^*\big\|
    \le\frac{9c\overline{\tau}\lambda\|\varepsilon\|_\infty}{8{\color{blue}n }}\sqrt{1.5s^*}$,
  and the desired result holds. So, it suffices to argue that $|S^{k-1}|\le1.5s^*$
  for all $k\in\mathbb{N}$. When $k=1$, the statement holds trivially since $w^{0}=0$
  implies $S^0=S^*$. Assuming that $|S^{k-1}|\le1.5s^*$ holds for $k=l$ with $l\ge 1$,
  we prove that it holds for $k=l+1$. Indeed,
  since $S^{l}\setminus S^*=\big\{i\notin S^*\!:w_i^{l}>\frac{1}{2}\big\}$,
  we have $w_i^{l}\in(\frac{1}{2},1]$ for $i\in S^{l}\setminus S^*$. Together with
  formula (3.3), we deduce that $\rho_l|\beta_i^{l}|\ge 1$, and hence
  the following inequality holds:
  \[
    \sqrt{|S^{l}\setminus S^*|}\le\sqrt{\sum_{i\in S^{l}\setminus S^*}\rho_l^2|\beta_i^{l}|^2}
    =\sqrt{\sum_{i\in S^{l}\setminus S^*}\rho_l^2|\beta_i^{l}-\beta_i^*|^2}.
  \]
  Since the statement holds for $k\!=l$, we get
  $\|\beta^{l}\!-\!\beta^*\|
  \le\!\frac{9c\overline{\tau}\lambda\|\varepsilon\|_\infty\sqrt{1.5s^*}}{8}$. So, it holds that
  \begin{equation}\label{Sl-equa}
    \sqrt{|S^{l}\setminus S^*|}\le \rho_l\|\beta^{l}-\beta^*\|
    \le \frac{9c\overline{\tau}\rho_l\lambda\|\varepsilon\|_\infty}{8}\sqrt{1.5s^*}
    \le \sqrt{0.5s^*}
  \end{equation}
  where the last inequality is due to $\rho_l\lambda\le\rho_3\lambda
  \le\frac{8}{9\sqrt{3}c\overline{\tau}\|\varepsilon\|_\infty}$.
  The inequality \eqref{Sl-equa} implies $|S^{l}|\leq1.5s^*$.
  This shows that the statement follows.
 \end{proof}

  To present the proof of Theorem 3, we need the following lemma
  which upper bounds $ \|v_{S^*}^{k}\|_{\infty}$, whose proof
  is given in Lemma 3 of \cite{TaoPanBi19}.
 \begin{lemma}\label{FLambdk}
  Let $F^k$ and $\Lambda^k$ be the index sets defined by (4.9). Then,
  \[
    \|v^k_{S^*}\|_{\infty}
    \le\max_{i\in S^*}\mathbb{I}_{\Lambda^k}(i)+\max_{i\in S^*}\mathbb{I}_{F^k}(i)
    \quad{\rm for\ each}\ k\in\{0\}\cup\mathbb{N}.
  \]
 \end{lemma}

 \noindent
 {\bf\large Proof of Theorem 3}:
 \begin{proof}
  For each $k\in\mathbb{N}$, define $S^{k-1}:= S^*\cup\{i\notin S^*\!: w^{k-1}_i>\frac{1}{2}\}$.
  Since the conclusion holds for $k=1$, it suffices to consider $k\ge 2$.
  By the proof of Theorem 2, $|S^{k-1}|\le1.5s^*$
  for all $k\in\mathbb{N}$. Moreover, by \eqref{Sl-equa} and $\rho_k\ge 1$,
  \begin{align}\label{equa-43}
   \sqrt{|S^{k-1}|}
   &=\sqrt{|S^{*}|+|S^{k-1}\setminus S^*|}
   \le\sqrt{s^*}+\sqrt{|S^{k-1}\setminus S^*|}\nonumber\\
  &\le\sqrt{s^*}+\big(2n^{-1}\overline{\tau}\|X\|_1+r_k\big)^{-1}
      \frac{\lambda\rho_{k-1}}{8}\big\|\beta^{k-1}-\beta^*\big\|
  \end{align}
  where the first inequality is due to $\sqrt{a+b}\le\sqrt{a}+\sqrt{b}$ for $a,b\ge0$,
  the last one is due to $\lambda\ge 16n^{-1}\overline{\tau}\|X\|_1+8r_k$.
  From \eqref{equa-41} and Lemma \ref{FLambdk}, we have
  \begin{align*}
   \|\beta^k-\beta^*\|
  &\le c\overline{\tau}\|\varepsilon\|_{\infty}\sqrt{|S^{k-1}|}
      \big[\lambda\big(\max_{i\in S^*}\mathbb{I}_{\Lambda^{k-1}}(i)+\max_{i\in S^*}\mathbb{I}_{F^{k-1}}(i)\big)\big]\\
  &\quad+c\overline{\tau}\|\varepsilon\|_{\infty}\sqrt{|S^{k-1}|}
       \big[2n^{-1}\overline{\tau}\|X\|_1+r_k\big]\\
  &\le c\overline{\tau}\|\varepsilon\|_{\infty}
      \Big[\lambda\sqrt{1.5s^*}\max_{i\in S^*}\mathbb{I}_{\Lambda^0}(i)
      +\lambda\sqrt{1.5s^*}\rho_{k-1}\|\beta^{k-1}\!-\beta^*\|\\
  &\quad +\big(2n^{-1}\overline{\tau}\|X\|_1+r_k\big)\sqrt{|S^{k-1}|}\Big]
  \end{align*}
 where the last inequality is since
 \(
 \max_{i\in S^*}\mathbb{I}_{F^{k-1}}(i)\leq \max_{i\in S^*}\rho_{k-1}\big||\beta_i^{k-1}|-|\beta_i^*|\big|\leq\rho_{k-1}\|\beta^{k-1}-\beta^*\|.
 \)
 Substituting \eqref{equa-43} into this inequality yields
 \begin{align*}
  \|\Delta\beta^{k}\|
  &\le c\overline{\tau}\|\varepsilon\|_{\infty}\sqrt{s^*}
    \big(2n^{-1}\overline{\tau}\|X\|_1+r_k\big)
     +c\overline{\tau}\lambda\|\varepsilon\|_{\infty}\sqrt{1.5s^*}\max_{i\in S^*}\mathbb{I}_{\Lambda^0}(i)\nonumber\\
  &\quad +c\overline{\tau}\|\varepsilon\|_{\infty}\rho_{k-1}\lambda(\sqrt{1.5s^*}+1/8)\|\beta^{k-1}-\beta^*\|\nonumber\\
  &\le 2cn^{-1}\overline{\tau}^2\|\varepsilon\|_{\infty}\sqrt{s^*}\!\|X\|_1
     + c\overline{\tau}\|\varepsilon\|_{\infty}\sqrt{s^*}r_k\nonumber\\
  &\quad +c\overline{\tau}\lambda\|\varepsilon\|_{\infty}\sqrt{1.5s^*}\max_{i\in S^*}\mathbb{I}_{\Lambda^0}(i)
         +\!\frac{\sqrt{3}}{3}\|\Delta\beta^{k-1}\|
 \end{align*}
  where the relation $\rho_{k-1}\lambda\le\rho_3\lambda \le [\sqrt{3}c\overline{\tau}\|\varepsilon\|_\infty(\sqrt{1.5s^*}+1/8)]^{-1}$ is used.
  The desired result follows by using the last recursion inequality.
 \end{proof}

  \noindent
 {\bf\large Appendix C}

 We describe the iterates of the semismooth Newton method
 and those of the semi-proximal ADMM in \cite{Gu16}.
 The iterates of the semismooth Newton method are as follows.
 \begin{algorithm}[h]
 \caption{\label{SNCG}{\bf\ \ A semismooth Newton method}}
 \textbf{Initialization:} Fix $k$ and $j$. Choose $0<c_1<c_2<1,\mu=10^{-5}$ and $u^0=0$.\\
 \textbf{while} the stopping conditions are not satisfied \textbf{do}
 \begin{enumerate}
  \item  Choose $U^l\in\mathcal{U}_j(u^l),V^l\in\mathcal{V}_j(u^l)$
         and set $W^l=\gamma_{2,j}^{-1}U^l+\gamma_{1,j}^{-1}XV^lX^{\mathbb{T}}$.
         Then, seek a solution $d^l\in\mathbb{R}^n$ to the following linear system
         \vspace{-0.3cm}
        \begin{equation}\label{SNCG-dj}
         (W^{l}+\mu I)d=-\Phi_{k,j}(u^l).
         \vspace{-0.3cm}
         \end{equation}
  \item Search the step-size $\alpha_l$ in the direction $d^l$ to satisfy
  \vspace{-0.3cm}
        \begin{align*}
         \Psi_{k,j}(u^l+\alpha_ld^l)\leq\Psi_{k,j}(u^l)+c_1\alpha_l\langle\nabla\Psi_{k,j}(u^l),d^l\rangle,\\
         |\langle\nabla\Psi_{k,j}(u^l+\alpha_ld^l),d^l\rangle|\le c_2|\langle\nabla\Psi_{k,j}(u^l),d^l\rangle|.
        \end{align*}
 \item Set $u^{l+1}=u^l+\alpha_ld^l$ and $l\leftarrow l+1$, and then go to Step 1.
 \end{enumerate}
 \vspace{-0.3cm}
 \textbf{end while}
 \end{algorithm}	

 Notice that the subproblem (3.1) can be equivalently written as
 \begin{equation}\label{Eweighted-L1}
  \min_{\beta\in\mathbb{R}^p,z\in\mathbb{R}^n}\Big\{f_{\tau}(z)+\|\omega^{k-1}\circ \beta\|_1
  \ \ {\rm s.t.}\ \ X\beta+z-y=0\Big\}
 \end{equation}
 whose dual problem, after an elementary calculation, takes the form of
 \begin{equation}\label{DEweighted-L1}
  \min_{u\in\mathbb{R}^n}\Big\{f_{\tau}^*(u)+\langle u,y\rangle
      \ \ {\rm s.t.}\ \ |(X^{\mathbb{T}}u)_i|\le \omega_i^{k-1},\ i=1,\ldots,p\Big\}.
 \end{equation}
 For a given $\sigma>0$, the augmented Lagrangian function of \eqref{Eweighted-L1}
 is given by
 \[
    L_{\sigma}(\beta,z,u):=f_{\tau}(z)+\|\omega^{k-1}\circ \beta\|_1
    +\langle u,X\beta+z-y\rangle+\frac{\sigma}{2}\|X\beta+z-y\|^2.
 \]
 The iterate steps of the semi-proximal ADMM in \cite{Gu18} are described as follows.
 \begin{algorithm}[h]
 \caption{\label{sPADMM}{\bf\ \ Semi-proximal ADMM for solving \eqref{Eweighted-L1}}}
 \textbf{Initialization:} Choose $\sigma>0,\gamma=\sigma\|X^{\mathbb{T}}X\|$ and
 $\varrho\in(1,\frac{\sqrt{5}+1}{2})$, and an initial point $(\beta^0,z^0,u^0)\in\mathbb{R}^p\times
 \mathbb{R}^n\times\mathbb{R}^n$ with $\beta^0=\beta^{k-1}$. Set $j=0$.\\
 \textbf{while} the stopping conditions are not satisfied \textbf{do}
 \begin{enumerate}
  \item  Compute the following convex minimization problem
         \begin{subequations}
         \begin{align}\label{ADMM-subprob1}
          \beta^{j+1}&=\mathop{\arg\min}_{\beta\in \mathbb{R}^{p}}L_{\sigma}(\beta,z^j,u^j)
          +\frac{1}{2}\|\beta-\beta^j\|_{\gamma I-\sigma X^{\mathbb{T}}X}^2,\\
         \label{ADMM-subprob2}
          z^{j+1}&=\mathop{\arg\min}_{z\in\mathbb{R}^{n}}L_{\sigma}(\beta^{j+1},z,u^j).
         \end{align}
         \end{subequations}
  \item Update the multiplier by $u^{j+1}=u^j+\varrho\sigma(X\beta^{j+1}+z^{j+1}-y)$.

  \item Set~$j\leftarrow j+1$, and then go to Step 1.
 \end{enumerate}
 \vspace{-0.3cm}
 \textbf{end while}
 \end{algorithm}
 \begin{remark}\label{remark-ADMM}
  {\bf(i)} Algorithm \ref{sPADMM} has a little difference from Algorithm 1 of \cite{Gu18}
  since here the semi-proximal term $\frac{1}{2}\|\beta-\beta^j\|_{\gamma I-\sigma X^{\mathbb{T}}X}^2$,
  rather than $\frac{1}{2}\|\beta-\beta^j\|_{\sigma(\gamma I-X^{\mathbb{T}}X)}^2$, is used.
  Let $h^j=\!\gamma\beta^j+\sigma X^{\mathbb{T}}(\!X\beta^j+z^j-y+u^j/\sigma)$.
  Problems \eqref{ADMM-subprob1} and \eqref{ADMM-subprob2}
  have a closed form solution:
  \begin{align*}
   \beta^{j+1}&={\rm sign}\big(\gamma^{-1}h^j\big)\max\big(|\gamma^{-1}h^j|-\gamma^{-1}\omega^{k-1},0\big),\\
   z^{j+1}&=\mathcal{P}_{\sigma^{-1}}f_{\tau}(y-X\beta^{j+1}-\sigma^{-1}u^j).
  \end{align*}

  \vspace{-0.3cm}
  \noindent
  {\bf(ii)} During our implementation of Algorithm \ref{sPADMM},
  we adjust $\sigma$ dynamically by the ratio of the primal and dual infeasibility.
  By comparing the first-order optimality conditions of \eqref{ADMM-subprob1} and
  \eqref{ADMM-subprob2} with those of \eqref{Eweighted-L1} and using the multiplier
  updating step, we measure the primal and infeasibility and
  the dual gap at $(\beta^{j},z^{j},u^{j})$ in terms of
  $\epsilon_{{\rm pinf}}^{j},\epsilon_{{\rm dinf}}^{j}$ and
  $\epsilon_{{\rm gap}}^{j}$, respectively:
  \begin{subequations}
   \begin{align}\label{dual-gap1}
    \epsilon_{{\rm dinf}}^{j}:=\frac{\sqrt{\|\zeta^j\|^2
   +\|(\varrho^{-1}\!-1)(u^{j}\!-\!u^{j-1})\|^2}}{1+\|y\|},\qquad\quad\\
   \label{dual-gap2}
   \epsilon_{{\rm pinf}}^{j}
   :=\frac{\|u^{j}-u^{j-1}\|}{\varrho\sigma(1+\|y\|)},\quad
   \epsilon_{{\rm gap}}^{j}\!:=\!\frac{|\omega_{\rm prim}^j+\omega_{\rm dual}^j|}
  {\max\big(1,0.5(\omega_{\rm prim}^j+\omega_{\rm dual}^j)\big)}
  \end{align}
  \end{subequations}
  where $\zeta^j\!:=X^{\mathbb{T}}(u^{j}\!-\!u^{j-1}\!-\!\sigma(X\beta^{j-1}\!-y+z^{j-1}))-\gamma(\beta^{j}\!-\!\beta^{j-1})$,
  and $\omega_{\rm prim}^j$ and $\omega_{\rm dual}^j$ are the objective
  values of \eqref{Eweighted-L1} and \eqref{DEweighted-L1}
  at $(\beta^{j},z^{j},u^{j})$. Different from \cite{Gu18}, when
  $\max(\epsilon_{{\rm pinf}}^{j},\epsilon_{{\rm dinf}}^{j},\epsilon_{{\rm gap}}^{j})
  \le\epsilon_{\rm ADMM}$ or $j>j_{\rm max}$, we terminate Algorithm \ref{sPADMM}.
  By comparing with the optimality conditions of \eqref{ADMM-subprob1}-\eqref{ADMM-subprob2}
  with those of \eqref{Eweighted-L1}, such a stopping criterion ensures that
  the obtained $(\beta^{j},z^{j},u^{j})$ is an approximate primal-dual solution pair.
  \end{remark}

 \noindent
 {\bf\large Appendix D}\\
 \noindent{\bf D.1. Performance comparisons of three solvers}

  We shall test the performance of MSCRA\_IPM, MSCRA\_ADMM and MSCRA\_PPA for
  computing the estimator $\widehat{\beta}$ in the same setting as in \cite{Fan14a}
  and \cite{Gu18}. Specifically, with $\beta^*\!=(2,\,0,\,1.5,\,0,\,0.8,\,0,\,0,\,1,\,0,\,1.75,\,0,\,0,\,0.75,\,0,\,0,\,0.3,\,{\bf 0}_{p-16}^{\mathbb{T}})^{\mathbb{T}}$ for $(p,n)=(1000,200)$,
  we obtain $n$ observations from (2.1), where the noise $\varepsilon$ comes
  from the distributions in \cite{Gu18}, including {\bf (1)} the normal distribution $N(0,2)$;
  {\bf (2)} the mixture normal distribution $0.9N(0,1)+0.1N(0,25)$,
  denoted by ${\rm MN}_1$; {\bf (3)} the mixture normal distribution
  $N(0,\sigma^2)$ with $\sigma\!\sim\!{\rm Unif}(1,5)$, denoted by ${\rm MN}_2$;
  {\bf (4)} the Laplace distribution with density $d(u)=0.5\exp(-|u|)$;
  {\bf (5)} the scaled Student's $t$-distribution with $4$ degrees of freedom
  $\sqrt{2}\times t_4$; and {\bf (6)} the Cauchy distribution with density
  $d(u)=\frac{1}{\pi(1+u^2)}$. For the covariance matrix $\Sigma_x$,
  we also consider those scenarios from \cite{Gu18}, including
  $\Sigma_x=I$; $\Sigma_x=(0.5^{|i-j|})_{ij}$ and $(0.8^{|i-j|})_{ij}$,
  denoted by ${\rm AR}_{0.5}$ and ${\rm AR}_{0.8}$; and
  $\Sigma_x\!=(\alpha+\!(1-\!\alpha)\mathbb{I}_{\{i=j\}})$ with $\alpha=0.5$
  and $0.8$, denoted by ${\rm CS}_{0.5}$ and ${\rm CS}_{0.8}$.
  We test the estimation and selection performance of the estimators computed
  with the solvers under each scenario in terms of the {\bf $\ell_2$-error},
  the CPU time, and the number of false positives (${\bf FP}$) and negatives (${\bf FN}$).

  As mentioned by \cite{Fan14a}, the cross-validation is not suitable for
  choosing the best $\nu=\lambda^{-1}$ due to the instability of $\ell_2$-error
  under heavy tails. We choose the best $\lambda$ by
 $
  \lambda_i=\max\big(0.01, \gamma_i\|X\|_1/n \big)
  \ \ {\rm with}\ \ \gamma_i=\gamma_{\rm min}+((i-1)/49)(\gamma_{\rm max}-\gamma_{\rm min})
 $
  by seeking the constant $\gamma$ optimally. Inspired by the choice strategy of $\lambda$
  in \cite{Fan14a}, we choose $\gamma$ based on 100 validation data-sets.
  Specifically, for each of data-sets, we ran a grid search to find the best
  $\gamma$ and then the best $\lambda$
  (with the lowest $\ell_2$-error of $\beta^f$) for the particular setting.
  The optimal $\gamma$ was recorded for each of the 100 validation data-sets.
  We denote by $\gamma_{\rm opt}$ the median of the 100 optimal $\gamma$,
  and use $\lambda=\max\big(0.01, \gamma_{\rm opt}\|X\|_1/n\big)$
  for the simulation studies. The best $\gamma$ is searched from
  $\gamma_1,\ldots,\gamma_{51}$ for $\gamma_{\rm min}=0.08$ and
  $\gamma_{\rm max}=0.38$. Such $\gamma_{\rm max}$ is
  such that $N_{\rm nz}(\beta^{f})$ attains or is close to $0$.

  Table \ref{result-1}-\ref{result-5} report the average {\bf $\ell_2$-error},
  ${\bf FP}$ and ${\bf FN}$ for $\tau\!=0.5$ and $0.75$ based on 100 simulations.
  For almost all test problems, MSCRA\_PPA requires only one-fifteenth of
  the CPU time of  MSCRA\_ADMM and MSCRA\_IPM, and its {\bf $\ell_2$-error}
  is comparable with that of MSCRA\_ADMM and MSCRA\_IPM.
  In addition, for all test problems, the ${\bf FP}$ of MSCRA\_PPA are
  lower than that of MSCRA\_ADMM and MSCRA\_IPM though its ${\bf FN}$ is
  a little higher than that of the latter two methods.
 \begin{table}[h]\tiny
  \caption{\small Estimation and selection performance of three solvers for $\Sigma_x=I$}
  \label{result-1}
  \centering
 \scalebox{0.7}{
 \begin{tabular}{cl|ccccc|cclcc}
  \hline
  \hline
 $\varepsilon$ & Method & $\gamma_{\rm opt}$ & $L_2$-error & FP  &  FN & Time(s) & $\gamma_{\rm opt}$&$L_2$-error & FP& FN &Time(s)\\
  &  & \multicolumn{4}{c}{$\tau=0.5$} &  & \multicolumn{5}{c}{$\tau=0.75$} \\
\hline
\multirow{3}*{$\mathcal{N}(0,2)$}
& IPM  & 0.104 & 0.444(0.107) & 5.100(2.057) & 0.730(0.468) & 4.221& 0.110 & 0.523(0.157) & 7.840(3.034) & 0.670(0.514) & 5.613\\
& ADMM & 0.104 & 0.446(0.106) & 5.100(2.028) & 0.730(0.468) & 3.033& 0.110 & 0.523(0.158) & 7.760(3.079) & 0.670(0.514) & 3.847\\
& PPA  & 0.116 & 0.446(0.119) & 1.920(1.228) & 0.800(0.426) & 0.138& 0.119 & 0.557(0.188) & 3.810(1.937) & 0.840(0.420) & 0.202\\
\hline
\multirow{3}*{${\rm MN_1}$}
& IPM  & 0.104 & 0.345(0.066) & 5.030(2.007) & 0.410(0.494) & 3.566& 0.110 & 0.377(0.078) & 6.860(2.741) & 0.490(0.502) & 4.168 \\
& ADMM & 0.104 & 0.345(0.067) & 5.150(2.110) & 0.410(0.494) & 2.601& 0.110 & 0.377(0.078) & 6.890(2.723) & 0.480(0.502) & 3.062 \\
& PPA  & 0.110 & 0.347(0.066) & 3.260(1.779) & 0.510(0.502) & 0.131& 0.116 & 0.375(0.061) & 5.050(2.333) & 0.590(0.494) & 0.191 \\
\hline
\multirow{3}*{${\rm MN_2}$}
& IPM  & 0.104 & 1.425(0.361) & 6.750(2.955) & 1.860(0.921) & 5.558& 0.122 & 1.764(0.501) & 4.220(2.377) & 2.660(1.085) & 5.568 \\
& ADMM & 0.104 & 1.427(0.356) & 6.760(3.114) & 1.880(0.902) & 3.829& 0.122 & 1.749(0.512) & 4.270(2.432) & 2.670(1.064) & 3.825\\
& PPA  & 0.116 & 1.347(0.343) & 2.480(1.823) & 2.320(0.994) & 0.133& 0.134 & 1.742(0.537) & 1.790(1.690) & 3.260(1.050) & 0.151 \\
 \hline
\multirow{3}*{ Laplace}
& IPM  & 0.098 & 0.324(0.071) & 7.410(2.775) & 0.220(0.416) & 3.835& 0.110 & 0.364(0.089) & 6.550(2.484) & 0.410(0.494) & 3.789 \\
& ADMM & 0.098 & 0.324(0.070) & 7.450(2.797) & 0.220(0.416) & 2.709& 0.110 & 0.365(0.089) & 6.580(2.458) & 0.400(0.492) & 2.761\\
& PPA  & 0.104 & 0.326(0.073) & 4.700(2.209) & 0.280(0.451) & 0.144& 0.116 & 0.382(0.094) & 4.970(2.158) & 0.480(0.502) & 0.204  \\
\hline
\multirow{3}*{${\rm \sqrt{2}\times t_4}$}
& IPM  & 0.104 & 0.487(0.139) & 5.330(2.301) & 0.760(0.474) & 4.677& 0.110 & 0.649(0.238) & 7.300(2.880) & 0.840(0.507) & 4.907 \\
& ADMM & 0.104 & 0.487(0.138) & 5.360(2.325) & 0.760(0.474) & 3.214& 0.110 & 0.647(0.239) & 7.360(2.812) & 0.840(0.507) & 3.340 \\
& PPA  & 0.110 & 0.502(0.180) & 3.160(1.587) & 0.790(0.478) & 0.157& 0.122 & 0.684(0.286) & 2.970(1.861) & 1.010(0.643) & 0.239  \\
\hline
\multirow{3}*{Cauchy}
& IPM  & 0.098 & 0.536(0.217) & 8.340(3.019) & 0.670(0.533) & 4.954& 0.110 & 0.730(0.364) & 6.740(2.493) & 1.000(0.765) & 5.488 \\
& ADMM & 0.098 & 0.531(0.216) & 8.340(2.879) & 0.680(0.530) & 2.989& 0.110 & 0.729(0.360) & 6.720(2.551) & 1.010(0.759) & 3.404\\
& PPA  & 0.116 & 0.560(0.274) & 1.780(1.203) & 0.910(0.637) & 0.166& 0.125 & 0.816(0.381) & 2.760(1.837) & 1.280(0.792) & 0.243\\
\hline
\hline
\end{tabular}}
\end{table}
\begin{table}[h]\tiny
 \caption{\small Estimation and selection performance of three solvers for ${\rm AR}_{0.5}$}\label{result-2}
\centering
\scalebox{0.7}{
\begin{tabular}{cl|ccccc|cclcc}
\hline
\hline
 $\varepsilon$ & Method & $\gamma_{\rm opt}$ & $L_2$-error & FP  &  FN & Time(s) & $\gamma_{\rm opt}$&$L_2$-error & FP& FN &Time(s)\\
  &  & \multicolumn{4}{c}{$\tau=0.5$} &  & \multicolumn{5}{c}{$\tau=0.75$} \\
\hline
\multirow{3}*{$\mathcal{N}(0,2)$}
& IPM  & 0.104 & 0.467(0.119) & 4.650(2.148) & 0.710(0.456) & 3.744&  0.110 & 0.609(0.222) & 6.830(2.843) & 0.800(0.512) & 4.312 \\
& ADMM & 0.104 & 0.474(0.120) & 4.620(2.112) & 0.730(0.446) & 2.553&  0.110 & 0.606(0.214) & 6.860(2.853) & 0.800(0.512) & 3.143 \\
& PPA  & 0.110 & 0.491(0.145) & 2.810(1.594) & 0.760(0.474) & 0.133&  0.122 & 0.591(0.199) & 3.020(1.664) & 0.870(0.442) & 0.201  \\
\hline
\multirow{3}*{${\rm MN_1}$}
& IPM  & 0.098 & 0.365(0.074) & 7.020(2.515) & 0.410(0.494) & 3.661& 0.110 & 0.399(0.076) & 6.450(2.679) & 0.570(0.498) & 3.729 \\
& ADMM & 0.098 & 0.367(0.073) & 7.070(2.536) & 0.400(0.492) & 2.746& 0.110 & 0.399(0.076) & 6.500(2.676) & 0.570(0.498) & 2.819 \\
& PPA  & 0.098 & 0.366(0.073) & 7.060(2.566) & 0.410(0.494) & 0.139&  0.122 & 0.423(0.127) & 3.390(1.959) & 0.630(0.485) & 0.180 \\
\hline
\multirow{3}*{${\rm MN_2}$}
& IPM  & 0.104 & 1.383(0.394) & 4.990(2.472) & 2.060(0.930) & 5.168& 0.122 & 1.665(0.434) & 3.640(2.013) & 2.610(0.920) & 5.339\\
& ADMM & 0.104 & 1.379(0.384) & 5.220(2.747) & 2.010(0.937) & 3.446& 0.122 & 1.679(0.420) & 3.670(2.080) & 2.590(0.911) & 3.764\\
& PPA  & 0.119 & 1.365(0.420) & 1.590(1.436) & 2.490(0.937) & 0.101&  0.131 & 1.705(0.512) & 2.100(1.755) & 3.010(0.959) & 0.167  \\
 \hline
\multirow{3}*{ Laplace}
& IPM  & 0.098 & 0.349(0.089) & 7.250(2.564) & 0.360(0.482) & 3.818& 0.110 & 0.381(0.099) & 6.320(2.624) & 0.580(0.496) & 4.513 \\
& ADMM & 0.098 & 0.349(0.089) & 7.250(2.591) & 0.360(0.482) & 2.851& 0.110 & 0.381(0.099) & 6.380(2.666) & 0.570(0.498) & 3.130\\
& PPA  & 0.104 & 0.352(0.088) & 4.600(2.079) & 0.410(0.494) & 0.125&  0.116 & 0.408(0.154) & 4.610(2.188) & 0.480(0.522) & 0.209 \\
\hline
\multirow{3}*{${\rm \sqrt{2}\times t_4}$}
& IPM  & 0.104 & 0.534(0.165) & 4.580(2.142) & 0.830(0.473) & 4.341& 0.110 & 0.734(0.291) & 6.920(2.990) & 1.070(0.573) & 5.785 \\
& ADMM & 0.104 & 0.533(0.165) & 4.590(2.109) & 0.830(0.473) & 3.179& 0.110 & 0.736(0.288) & 6.860(3.052) & 1.070(0.573) & 3.891 \\
& PPA  & 0.110 & 0.542(0.180) & 3.020(1.723) & 0.860(0.472) & 0.129& 0.122 & 0.710(0.283) & 3.240(1.782) & 1.150(0.575) & 0.209 \\
  \hline
\multirow{3}*{Cauchy}
& IPM  & 0.101 & 0.544(0.245) & 6.130(2.232) & 0.820(0.539) & 4.912& 0.104 & 0.695(0.343) & 9.450(3.105) & 0.980(0.681) & 5.948 \\
& ADMM & 0.104 & 0.538(0.258) & 4.890(2.136) & 0.860(0.513) & 2.952& 0.104 & 0.693(0.335) & 9.530(2.883) & 0.950(0.672) & 3.686 \\
& PPA  & 0.116 & 0.561(0.280) & 1.740(1.292) & 0.980(0.603) & 0.169&  0.122 & 0.879(0.473) & 3.270(1.814) & 1.430(0.956) & 0.233  \\
\hline
\hline
 \end{tabular}}
\end{table}
\begin{table}[h]\tiny
 \caption{\small Estimation and selection performance of three solvers for ${\rm AR}_{0.8}$}
 \label{result-3}
\centering
\scalebox{0.7}{
\begin{tabular}{cl|ccccc|cclcc}
\hline
\hline
 $\varepsilon$ & Method & $\gamma_{\rm opt}$ & $L_2$-error & FP  &  FN & Time(s) & $\gamma_{\rm opt}$&$L_2$-error & FP& FN &Time(s)\\
  &  & \multicolumn{4}{c}{$\tau=0.5$} &  & \multicolumn{5}{c}{$\tau=0.75$} \\
\hline
\multirow{3}*{$\mathcal{N}(0,2)$}
& IPM  & 0.095 & 0.852(0.361) & 7.050(2.504) & 1.260(0.733) & 4.117& 0.098 & 0.986(0.408) & 10.740(3.852) & 1.400(0.804) & 6.170 \\
& ADMM & 0.092 & 0.835(0.336) & 8.800(2.723) & 1.240(0.698) & 3.306& 0.098 & 0.996(0.404) & 10.940(3.961) & 1.400(0.816) & 4.721 \\
& PPA  & 0.110 & 0.910(0.404) & 2.390(1.550) & 1.520(0.731) & 0.111&  0.110 & 0.965(0.387) & 5.140(2.454) & 1.440(0.701) & 0.193 \\
\hline
\multirow{3}*{${\rm MN_1}$}
& IPM  & 0.098 & 0.530(0.208) & 5.300(2.368) & 0.780(0.504) & 3.683& 0.098 & 0.622(0.254) & 9.510(4.036) & 0.850(0.557) & 5.205 \\
& ADMM & 0.092 & 0.519(0.184) & 8.460(2.844) & 0.770(0.489) & 2.933& 0.098 & 0.625(0.261) & 9.630(4.099) & 0.850(0.557) & 3.851 \\
& PPA  & 0.104 & 0.550(0.227) & 3.550(1.977) & 0.800(0.512) & 0.132&  0.110 & 0.644(0.321) & 5.120(2.363) & 1.000(0.682) & 0.184 \\
\hline
\multirow{3}*{${\rm MN_2}$}
& IPM  & 0.104 & 1.742(0.616) & 4.350(2.086) & 2.590(0.889) & 4.362& 0.122 & 2.113(0.641) & 3.120(1.981) & 3.020(0.995) & 5.187 \\
& ADMM & 0.104 & 1.713(0.642) & 4.560(2.203) & 2.500(0.959) & 3.187& 0.116 & 2.139(0.629) & 4.230(2.155) & 2.970(0.958) & 4.269 \\
& PPA  & 0.140 & 1.809(0.649) & 0.820(0.936) & 2.920(0.929) & 0.085&  0.152 & 2.125(0.721) & 0.940(0.886) & 3.290(0.868) & 0.126 \\
 \hline
\multirow{3}*{ Laplace}
& IPM  & 0.098 & 0.520(0.257) & 5.810(2.639) & 0.720(0.637) & 3.767& 0.104 & 0.650(0.375) & 6.980(3.291) & 0.980(0.710) & 3.990 \\
& ADMM & 0.098 & 0.510(0.242) & 5.880(2.626) & 0.710(0.608) & 2.864& 0.104 & 0.645(0.370) & 7.140(3.333) & 0.970(0.703) & 3.180 \\
& PPA  & 0.104 & 0.543(0.267) & 3.780(2.177) & 0.840(0.615) & 0.124&  0.116 & 0.679(0.386) & 3.710(2.176) & 1.150(0.716) & 0.167\\
\hline
\multirow{3}*{${\rm \sqrt{2}\times t_4}$}
& IPM  & 0.095 & 0.955(0.412) & 7.180(2.754) & 1.470(0.658) & 4.517& 0.098 & 1.135(0.465) & 10.250(4.029) & 1.660(0.831) & 5.201\\
& ADMM & 0.092 & 0.934(0.407) & 8.700(3.125) & 1.410(0.653) & 3.236& 0.098 & 1.135(0.485) & 10.400(3.929) & 1.660(0.867) & 3.641\\
& PPA  & 0.110 & 1.009(0.400) & 2.570(1.736) & 1.630(0.646) & 0.118&  0.110 & 1.190(0.542) & 5.450(2.516) & 1.870(0.939) & 0.194\\
  \hline
\multirow{3}*{Cauchy}
& IPM  & 0.104 & 0.891(0.452) & 3.440(2.134) & 1.420(0.684) & 3.853& 0.110 & 1.168(0.573) & 4.970(2.676) & 1.790(0.946) & 4.842\\
& ADMM & 0.098 & 0.850(0.435) & 5.590(2.586) & 1.320(0.723) & 2.672&  0.110 & 1.153(0.549) & 4.950(2.668) & 1.770(0.908) & 2.901\\
& PPA  & 0.116 & 0.962(0.452) & 1.380(1.237) & 1.570(0.700) & 0.157&  0.122 & 1.138(0.570) & 2.920(1.895) & 1.800(0.921) & 0.205 \\
\hline
\hline
 \end{tabular}}
\end{table}
\begin{table}[h]\tiny
 \caption{\small Estimation and selection performance of three solvers for ${\rm CS}_{0.5}$}
 \label{result-4}
\centering
\scalebox{0.7}{
\begin{tabular}{cl|ccccc|cclcc}
\hline
\hline
 $\varepsilon$ & Method & $\gamma_{\rm opt}$ & $L_2$-error & FP  &  FN & Time(s) & $\gamma_{\rm opt}$&$L_2$-error & FP& FN &Time(s)\\
  &  & \multicolumn{4}{c}{$\tau=0.5$} &  & \multicolumn{5}{c}{$\tau=0.75$} \\
\hline
\multirow{3}*{$\mathcal{N}(0,2)$}
& IPM  & 0.092 & 0.683(0.266) & 1.710(1.597) & 1.130(0.464) & 3.819& 0.092 & 0.943(0.366) & 3.810(2.759) & 1.340(0.685) & 4.533 \\
& ADMM & 0.092 & 0.700(0.272) & 1.750(1.459) & 1.140(0.472) & 3.336& 0.098 & 0.962(0.388) & 2.780(2.245) & 1.450(0.757) & 3.761 \\
& PPA  & 0.104 & 0.744(0.282) & 0.650(0.880) & 1.260(0.543) & 0.195&  0.116 & 0.934(0.347) & 1.020(1.163) & 1.580(0.684) & 0.227 \\
\hline
\multirow{3}*{${\rm MN_1}$}
& IPM  & 0.092 & 0.437(0.093) & 1.300(1.243) & 0.810(0.394) & 3.366& 0.098 & 0.505(0.157) & 2.070(1.816) & 0.840(0.368) & 3.687\\
& ADMM & 0.098 & 0.441(0.097) & 0.730(0.777) & 0.820(0.386) & 2.981& 0.098 & 0.506(0.148) & 2.030(1.702) & 0.840(0.368) & 3.475 \\
& PPA  & 0.104 & 0.448(0.107) & 0.350(0.557) & 0.930(0.293) & 0.178&  0.116 & 0.523(0.192) & 0.420(0.867) & 1.020(0.200) & 0.235 \\
\hline
\multirow{3}*{${\rm MN_2}$}
& IPM  & 0.110 & 1.919(0.526) & 2.320(1.999) & 3.090(0.877) & 3.447&  0.122 & 2.253(0.492) & 2.690(1.813) & 3.550(0.744) & 3.224 \\
& ADMM & 0.122 & 1.977(0.490) & 3.210(2.271) & 3.100(0.882) & 3.088& 0.143 & 2.268(0.451) & 3.800(2.094) & 3.530(0.745) & 3.241 \\
& PPA  & 0.152 & 2.016(0.545) & 1.650(1.480) & 3.410(0.866) & 0.117&  0.155 & 2.444(0.579) & 2.600(1.717) & 3.830(0.842) & 0.170 \\
 \hline
\multirow{3}*{ Laplace}
& IPM  & 0.086 & 0.445(0.140) & 2.390(2.117) & 0.810(0.394) & 3.926& 0.098 & 0.568(0.253) & 2.290(2.027) & 1.010(0.414) & 3.868 \\
& ADMM & 0.086 & 0.445(0.139) & 2.520(2.134) & 0.800(0.402) & 3.773& 0.092 & 0.559(0.212) & 3.480(2.552) & 0.920(0.442) & 3.889 \\
& PPA  & 0.098 & 0.469(0.167) & 0.930(1.380) & 0.910(0.379) & 0.181&  0.104 & 0.586(0.279) & 1.570(2.171) & 1.110(0.510) & 0.250 \\
\hline
\multirow{3}*{${\rm \sqrt{2}\times t_4}$}
& IPM  & 0.092 & 0.874(0.352) & 1.960(1.780) & 1.400(0.651) & 4.345&  0.092 & 1.206(0.486) & 4.150(2.724) & 1.710(0.868) & 4.657 \\
& ADMM & 0.086 & 0.905(0.339) & 3.600(2.229) & 1.310(0.598) & 4.071& 0.095 & 1.259(0.448) & 3.760(2.527) & 1.800(0.791) & 3.875 \\
& PPA  & 0.110 & 0.966(0.347) & 0.910(1.215) & 1.610(0.680) & 0.165&  0.116 & 1.172(0.429) & 1.290(1.241) & 1.980(0.816) & 0.216 \\
  \hline
\multirow{3}*{Cauchy}
& IPM  & 0.086 & 0.803(0.377) & 3.050(2.208) & 1.330(0.620) & 5.123& 0.092 & 1.239(0.575) & 3.910(2.016) & 1.900(0.859) & 5.142\\
& ADMM & 0.092 & 0.896(0.436) & 2.270(1.869) & 1.480(0.674) & 3.599&  0.095 & 1.392(0.592) & 4.190(2.608) & 2.040(0.887) & 3.471 \\
& PPA  & 0.101 & 0.880(0.415) & 1.200(1.198) & 1.460(0.658) & 0.278&  0.113 & 1.237(0.502) & 1.470(1.540) & 2.030(0.834) & 0.333 \\
\hline
\hline
 \end{tabular}}
\end{table}
\begin{table}[h]\tiny
 \caption{\small Estimation and selection performance of three solvers for ${\rm CS}_{0.8}$}\label{result-5}
\centering
\scalebox{0.7}{
\begin{tabular}{cl|ccccc|cclcc}
\hline
\hline
 $\varepsilon$ & Method & $\gamma_{\rm opt}$ & $L_2$-error & FP  &  FN & Time(s) & $\gamma_{\rm opt}$&$L_2$-error & FP& FN &Time(s)\\
  &  & \multicolumn{4}{c}{$\tau=0.5$} &  & \multicolumn{5}{c}{$\tau=0.75$} \\
\hline
\multirow{3}*{$\mathcal{N}(0,2)$}
& IPM  & 0.092 & 1.572(0.411) & 1.020(1.263) & 2.630(0.761) & 2.879& 0.098 & 1.803(0.469) & 1.480(1.337) & 2.890(0.840) & 2.907 \\
& ADMM & 0.131 & 1.683(0.365) & 2.050(1.617) & 2.820(0.796) & 2.979& 0.116 & 1.923(0.462) & 3.050(2.057) & 2.950(0.903) & 3.077 \\
& PPA  & 0.140 & 1.709(0.423) & 0.650(1.029) & 3.010(0.759) & 0.229&  0.140 & 1.939(0.460) & 1.210(1.233) & 3.220(0.773) & 0.177  \\
\hline
\multirow{3}*{${\rm MN_1}$}
& IPM  & 0.086 & 0.971(0.339) & 0.330(0.604) & 1.750(0.657) & 3.269& 0.086 & 1.118(0.405) & 0.700(0.835) & 1.840(0.762) & 3.355 \\
& ADMM & 0.086 & 0.952(0.363) & 0.910(1.173) & 1.600(0.696) & 3.178& 0.098 & 1.249(0.365) & 1.620(1.523) & 1.980(0.738) & 3.230 \\
& PPA  & 0.110 & 1.128(0.336) & 0.110(0.314) & 2.070(0.655) & 0.202&  0.110 & 1.283(0.392) & 0.460(0.784) & 2.270(0.777) & 0.150 \\
\hline
\multirow{3}*{${\rm MN_2}$}
& IPM  & 0.134 & 3.087(0.643) & 3.890(2.331) & 4.510(0.893) & 2.683&  0.125 & 3.371(0.602) & 4.780(2.729) & 4.910(0.911) & 2.739 \\
& ADMM & 0.137 & 2.897(0.496) & 7.840(3.589) & 4.250(0.903) & 3.432& 0.134 & 3.197(0.477) & 8.640(3.586) & 4.600(0.964) & 3.491 \\
& PPA  & 0.158 & 3.161(0.681) & 3.910(2.708) & 4.680(0.898) & 0.146&  0.149 & 3.507(0.625) & 4.710(2.467) & 5.120(0.868) & 0.117  \\
 \hline
\multirow{3}*{ Laplace}
& IPM  & 0.086 & 1.066(0.409) & 0.380(0.708) & 1.910(0.753) & 3.352& 0.086 & 1.372(0.493) & 1.130(1.284) & 2.350(0.903) & 3.417\\
& ADMM & 0.098 & 1.177(0.441) & 1.350(1.591) & 2.010(0.745) & 3.248& 0.104 & 1.540(0.494) & 2.510(2.267) & 2.510(0.904) & 3.223 \\
& PPA  & 0.110 & 1.254(0.427) & 0.220(0.561) & 2.350(0.783) & 0.192&  0.128 & 1.558(0.496) & 0.710(0.977) & 2.800(0.829) & 0.157  \\
\hline
\multirow{3}*{${\rm \sqrt{2}\times t_4}$}
& IPM  & 0.101 & 1.795(0.435) & 1.300(1.314) & 2.940(0.789) & 2.923& 0.104 & 2.160(0.517) & 2.280(1.735) & 3.230(0.827) & 2.980 \\
& ADMM & 0.128 & 1.889(0.409) & 3.320(2.344) & 2.920(0.813) & 3.215& 0.110 & 2.210(0.462) & 5.180(3.439) & 3.250(0.833) & 3.345 \\
& PPA  & 0.146 & 1.923(0.454) & 1.150(1.507) & 3.200(0.816) & 0.166& 0.152 & 2.261(0.547) & 1.580(1.505) & 3.570(0.807) & 0.137 \\
  \hline
\multirow{3}*{Cauchy}
& IPM  & 0.095 & 1.986(0.618) & 1.560(1.486) & 3.230(0.874) & 3.267& 0.113 & 2.498(0.734) & 2.390(1.933) & 3.850(1.019) & 3.122 \\
& ADMM & 0.128 & 2.181(0.564) & 4.210(2.552) & 3.440(0.903) & 2.870& 0.116 & 2.417(0.587) & 5.240(3.108) & 3.630(1.012) & 2.881 \\
& PPA  & 0.158 & 2.357(0.700) & 1.460(1.374) & 3.800(0.888) & 0.212& 0.134 & 2.667(0.805) & 2.650(2.167) & 4.160(1.080) & 0.178 \\                                                           \hline
\hline
 \end{tabular}}
\end{table}


 \noindent{\bf D.2. Performance on a real data example}

  Now we test the performance of MSCRA\_PPA on a real data set from
  \url{https://www.ncbi.nlm.nih.gov}, which is used by \cite{Scheetz06}
  to illustrate the gene regulation in mammalian eyes
  and to gain insight into genetic variation related to human eyes.
  This microarray data comprises gene expression levels of $31,042$ probes
  on $120$ twelve-week-old laboratory rats. For the $31,042$ probes,
  as suggested by \cite{Scheetz06}, we first carry out the preprocessing
  to obtain $18,986$ probes. Among those probes,
  there is one probe, 1389163$\_$at, corresponding to gene TRIM32, that was found to be
  associated with the Bardet-Biedl syndrome (see \cite{Chiang06}). We are interested in how
  the expression of this gene depends on the expressions of all other 18,985 genes.
  To achieve this goal, we select 3,000 probes with the largest variances and
  then standardize the selected 3,000 probes such that they have mean $0$ and
  standard deviation $1$, as \cite{Gu16} and \cite{Wang12} did.
  Thus, we obtain an $n\times p$ sample matrix $X'$ with $n=120$ and $p=3000$,
  and set $X=[e\ \ X']\in\mathbb{R}^{n\times(p+1)}$.

  Since the previous numerical tests show that MSCRA\_IPM
  and MSCRA\_ADMM have very similar performance, we use MSCRA\_PPA and
  MSCRA\_ADMM with $\tau=0.25, 0.5$ and $0.75$
  to analyze the data on all 120 rats. The parameter $\nu=\lambda^{-1}$ is
  used with $\lambda=\max\big(0.01, \gamma\|X\|_1/n\big)$, where
  $\gamma$ is selected via five-fold cross-validation.
  The results are reported on the third and fourth columns of Table \ref{T-6}.
  We also conduct 50 random partitions on the data,
  each of which has 80 rats in the training set and 40 rats in the validation set.
  We apply MSCRA\_ADMM and MSCRA\_PPA to the training set with $\lambda$ chosen
  as above and evaluate its prediction error on the validation set by calculating
  $\frac{1}{40}\sum_{i\in {\rm validation}}\theta_{\tau}(y_i-\beta_0-x_i^\mathbb{T}\widehat{\beta}^f)$,
  where $x_i^\mathbb{T}$ means the $i$th row of $X'$. The average number of
  selected genes, prediction errors and times over the $50$ partitions
  are listed in the last three columns of Table \ref{T-6}.
  We see that the average number of the genes selected by MSCRA\_PPA is less
  than that of the genes selected by MSCRA\_ADMM, the average prediction error
  of the former is lower than that of the latter, and the average CPU time of
  the former is about one-fifteenth of the latter.

 \begin{table}[h]\tiny
  \caption{\small Analysis of the microarray data by MSCRA\_PPA and MSCRA\_ADMM}\label{T-6}
  \centering
  \begin{tabular}{cc|cc|ccc}
   \hline
   \hline
   \multirow{2}*{Method}& \multirow{2}*{$\tau$} &\multicolumn{2}{c|}{All data}&  \multicolumn{3}{c}{Random partition} \\
   \cline{5-7}
  & &   $\#$genes &Time(s) & Ave.$\#$genes  &Pre\_error &Time(s)\\
   \hline
  \multirow{3}*{ADMM} & 0.25 & 17 & 3.843   & 17.200(1.807)        &0.050(0.009)   &4.686(0.804)   \\
                      & 0.5  & 27 & 4.141   & 20.960(4.323)        &0.029(0.005)   &3.555(0.496)   \\
                      & 0.75 & 19 & 4.314   & 21.280(2.611)        &0.040(0.005)   &3.534(0.405)     \\
 \hline
 \multirow{3}*{PPA}   & 0.25 & 20 & 0.208    &  16.440(3.721)        &0.023(0.006)   &0.235(0.056) \\
                      & 0.5  & 27 & 0.226    &  20.740(4.237)        &0.029(0.005)   &0.247(0.136)   \\
                      & 0.75 & 17 & 0.181    &  12.500(3.032)        &0.024(0.004)   &0.352(0.068)   \\
 \hline
 \hline
 \end{tabular}
 \end{table}	
 \end{document}